\documentclass[a4paper]{amsart}
\usepackage{hyperref}
\usepackage{txfonts, amsmath,amstext,amsthm,amscd,amsopn,verbatim,amssymb, amsfonts}
\usepackage{fullpage}

\usepackage[bbgreekl]{mathbbol}
\usepackage{tikz}
\usetikzlibrary{matrix}
\usetikzlibrary{shapes}
\usetikzlibrary{arrows}  
\usetikzlibrary{calc,3d}
\usetikzlibrary{decorations,decorations.pathmorphing}
\usetikzlibrary{through}
\tikzset{ext/.style={circle, draw,inner sep=1pt},int/.style={circle,draw,fill,inner sep=1pt},nil/.style={inner sep=1pt}}
\tikzset{exte/.style={circle, draw,inner sep=3pt},inte/.style={circle,draw,fill,inner sep=3pt}}
\tikzset{diagram/.style={matrix of math nodes, row sep=3em, column sep=2.5em, text height=1.5ex, text depth=0.25ex}}
\tikzset{diagram2/.style={matrix of math nodes, row sep=0.5em, column sep=0.5em, text height=1.5ex, text depth=0.25ex}}
\theoremstyle{plain}
  \newtheorem{thm}{Theorem}
   \newtheorem*{thm*}{Theorem}
  
  \newtheorem{prop}{Proposition}
  
  \newtheorem{cor}[prop]{Corollary}
  \newtheorem{lemma}{Lemma}
\theoremstyle{definition}
  \newtheorem{ex}{Example}
  \newtheorem{rem}{Remark}
  \newtheorem*{con*}{Conjecture}

\newcommand{\Hom}{\mathrm{Hom}}

\newcommand{\R}{{\mathbb{R}}}

\newcommand{\K}{{\mathbb{K}}}

\newcommand{\HGC}{{\mathrm{HGC}}}
\newcommand{\fHGC}{{\mathrm{fHGC}}}

\newcommand{\Graphs}{{\mathsf{Graphs}}}
\newcommand{\fGraphs}{{\mathsf{fGraphs}}}

\newcommand{\mF}{{\mathcal{F}}}
\newcommand{\gr}{{\mathit{gr}}}

\newcommand{\Gra}{{\mathsf{Gra}}}

\newcommand{\Def}{\mathrm{Def}}

\newcommand{\Fund}{\mathrm{Fund}}

\newcommand{\op}{\mathcal}

\newcommand{\Br}{\mathsf{Br}}

\newcommand{\Lie}{\mathsf{Lie}}

\newcommand{\hoLie}{\mathsf{hoLie}}

\newcommand{\Poiss}{\mathsf{Poiss}}
\newcommand{\hoPoiss}{\mathsf{hoPoiss}}
\newcommand{\Com}{\mathsf{Com}}

\newcommand{\FM}{\mathsf{FM}}

\newcommand{\Der}{\mathrm{Der}}

\newcommand{\bbS}{\mathbb{S}}
\newcommand{\ICG}{\mathrm{ICG}}

\newcommand{\bpm}{\begin{pmatrix}}
\newcommand{\epm}{\end{pmatrix}}

\newcommand{\GC}{\mathrm{GC}}

\newcommand{\fGC}{\mathrm{fGC}}

\newcommand{\hoe}{\mathsf{hoe}}
\newcommand{\e}{\mathsf{e}}

\newcommand{\lo}{\longrightarrow}

\newcommand{\gra}{\mathrm{gra}}

\newcommand{\stG}{{}^*{\Graphs}}

\newcommand{\bbo}{{\mathbb{1}}}

\newcommand{\calC}{\mathcal{C}}
\newcommand{\hDer}{\mathrm{hDer}}
\newcommand{\hIBim}{\operatorname{hIBim}}
\newcommand{\hOperad}{\operatorname{hOperad}}

\newcommand{\Diff}{\mathit{Diff}}

\begin{document}
\title{Relative (non-)formality of the little cubes operads and \\ the algebraic  Cerf lemma}
\author{Victor Turchin}
\address{Department of Mathematics\\
  Kansas State University\\
  138 Cardwell Hall\\
  Manhatan, KS 66506, USA}
  \email{turchin@ksu.edu}
\author{Thomas Willwacher}
\address{Institute of Mathematics \\ University of Zurich \\  
Winterthurerstrasse 190 \\
8057 Zurich, Switzerland}
\email{thomas.willwacher@math.uzh.ch}

\thanks{V.T. acknowledges Max-Planck-Institut f\"ur Mathematik  (Bonn) and the Institut des Hautes Etudes Scientifique for hospitality}
\thanks{T.W. acknowledges partial support by the Swiss National Science Foundation (grant 200021\_150012 and the SwissMap NCCR)}

 


\begin{abstract}
It is shown that the operad maps $E_n\to E_{n+k}$ are formal over the reals for $k\geq 2$ and non-formal for $k=1$. Furthermore we compute the homology of the deformation complex of the operad maps $E_{n}\to E_{n+1}$, proving an algebraic version of the Cerf lemma. 
\end{abstract}

\maketitle

\section{Introduction}
We consider the operads of chains $E_n$ of the little $n$-cubes operads $\calC_n$.
There are natural embeddings $\calC_n\to \calC_{n+k}$ for $k\geq 1$, and hence operad maps $E_n\to E_{n+k}$.
They induce maps in homology
\[
\e_n := H(E_n) \to H(E_{n+k}) =: \e_{n+k}.
\]
The operad $\e_n$ is generated by the two generators of $H(E_n(2))\cong H(S^{n-1})$, the degree zero generator denoted by $\wedge$ and the degree $n-1$ generator being denoted by $[,]$.
The map in homology $\e_n\to \e_{n+k}$ above is obtained by sending the product generator to the product and the bracket to zero. 

A quasi-isomorphism of operad maps $f:\op P \to \op Q$, $f':\op P'\to \op Q'$ is a commutative diagram 
\[
\begin{tikzpicture}
\matrix[diagram](m) { \op P & \op Q \\ \op P' & \op Q' \\ };
\draw[-latex] (m-1-1) edge node[above] {$f$}  (m-1-2)
		    (m-1-1) edge node[left] {$\simeq$}  (m-2-1)
		    (m-1-2) edge node[right] {$\simeq$} (m-2-2)
		    (m-2-1) edge node[above] {$f'$} (m-2-2);
\end{tikzpicture}
\]
in which the vertical maps are quasi-isomorphisms. Two maps $f$ and $f'$ are called quasi-isomorphic if they can be related to each other by a zigzag of quasi-isomorphisms. The operad map $f$ is called formal if it is quasi-isomorphic to the induced map $H(f)$ on homology.  

We show the following result.
\begin{thm}\label{thm:main}
The map $E_n\to E_{n+k}$ is formal over $\R$ for $k\geq 2$ and non-formal over $\R$ for $k=1$.
\end{thm}
In particular, one finds that the $E_2$ operad is not formal as a multiplicative operad.

Theorem \ref{thm:main} has been shown for $k> n$ by P. Lambrechts and I. Voli\'c \cite{LV}. They  notice that Kontsevich\rq{}s proof of the formality of $E_n$~\cite{K2} which  uses the Fulton-MacPherson model $\FM_n$ for $\calC_n$, graph-complexes and semi-algebraic forms, can be also adapted to study the relative formality. The main point in their argument is that the restriction from $\FM_{n+k}$ to $\FM_n$ of the semi-algebraic differential forms corresponding to graphs is zero by degree reasons for almost all forms. In~\cite{LV}  the Kontsevich construction is reproduced in full detail. The construction uses the theory of semi-algebraic differential forms, which was only sketched by Kontsevich and Soibelman in~\cite{KS} and is developed in more detail in~\cite{HLTV}. Theorem~\ref{thm:main} completely solves the relative formality problem of the little cubes operads over~$\R$. Also note that for $k=0$ one obtains the identity map $E_n\to E_n$, which is formal by the formality of the little $n$-cubes operad. 

The main motivation of Lambrechts and Voli\'c to prove the relative
formality was in its application to the embedding calculus. From the
improved range of formality given by Theorem 1 it follows that that the
spectral sequence associated with the Goodwillie-Weiss calculus and
computing the homology of the space of smooth embeddings $Emb(M^m,\R^n)$
of an $m$-manifold into $\R^n$  collapses rationally at the second term
whenever $n\geq 2m+2$, the condition already required for the limit of the
tower $T_\infty C_* Emb(M,\R^n)$ to have the same homology as
$Emb(M^m,\R^n)$. In particular all the results of [1] are improved to this
range versus $n\geq 2E(M)+1$ as stated in [1] (where $E(M)$ is  the
smallest dimension of a Euclidean space in which $M$ can be embedded).

We remark that another codimension one non-formality result was recently discovered by M. Livernet, who proved the non-formality of the Swiss Cheese operads~\cite{livernet}. It does not seem however that her result implies ours or vice versa. Also her approach is very different from ours, using the non-vanishing of operadic Massey products.

Our proof of Theorem~\ref{thm:main} is a more careful analysis of the Kontsevich-Lambrecht-Voli\'c construction. In case $k\geq 2$ it is obtained by a more careful degree-counting of the forms which again implies vanishing of the forms obtained by restriction. For the case $k=1$, on the contrary we show that  some forms on $\FM_{n+1}$ associated to certain graphs do not vanish when restricted on $\FM_{n}$. The latter fact together with a careful application of the deformation theory of operad maps proves the non-formality.

We furthermore study the deformation theory of the above operad map in codimension $k=1$.
Note that (for any $k$) the deformation complexes of the operad maps $E_n\to E_{n+k}$ carry a cup product and come equipped with a map from the (homotopy) derivations of $E_{n+k}$.

\begin{thm}[Algebraic version of the Cerf lemma]\label{thm:schoenflies}
Over $\R$, the homology of the deformation complex of the operad map $E_n\to E_{n+1}$ is generated by the images of the homotopy derivations of $E_{n+1}$ under the cup product.
\end{thm}
A more precise statement can be found as Theorem \ref{thm:PhiDef} below. This theorem shows the rigidity of the deformations of $E_n$ inside $E_{n+1}$. We call it the {\it Algebraic Cerf Lemma} because of the connection to the study of the spaces of embeddings.\footnote{In the first draft of this paper  we called this result {\it Algebraic Schoenflies Theorem}, because of this connection and because the (generalized) Schoenflies theorem also indicates the rigidity for codimension one embeddings. But later  we found  another result due to Cerf  that fits better the picture by detecting the aforementioned rigidity on the  level of spaces.}  Let $\Diff_\partial(D^{n})$ denote the group of diffeomorphisms of the $n$-disc preserving the boundary pointwise, and let $Emb_\partial(D^n,D^{n+k})$ be the space of smooth embeddings $D^n\hookrightarrow D^{n+k}$ of discs with the presribed behavior at the boundary. 
Cerf proved that  the natural scanning map
\begin{equation}
\Diff_\partial(D^{n+1})\to \Omega Emb_\partial(D^n,D^{n+1})
\label{eq:cerf}
\end{equation}
is a weak equivalence~\cite[Appendix, Section 5]{Cerf2}\footnote{This appendix was published earlier as~\cite{Cerf1}.}, see also~\cite[Proposition~5.3]{Budney} where this result is stated in the way we present it here.
%
Our Theorem~\ref{thm:schoenflies} can be  interpreted as a similar property of the limits of the Goodwillie-Weiss towers for the singular chains of $\Diff_\partial(D^{n+1})$ and $Emb_\partial(D^n,D^{n+1})$, see Section~\ref{s:emb_calcl}. 
The fact that the deformation theory of operads and the manifold functor calculus can detect this codimension one rigidity for spaces of embeddings gives a hope that the Goodwilie-Weiss calculus can be (with a certain care of course) applied to the study of  codimension zero and one embeddings.

For general $k\geq 1$ the deformation theory of the homology maps $\e_n\to \e_{n+k}$ has been studied  in \cite{Turchin3,LambrechtsTurchin,Turchin1}, where the homology of the resulting deformation complex is described in terms of the graph homology. In this case the Cerf Lemma does not hold, there are many additional classes beyond those originating in the homotopy automorphisms of $\e_{n+k}$. In a follow-up paper~\cite{TW} we continue the study of these complexes  and interpret them  as the Kontsevich type graph-complexes decorated in  non-trivial representations of the groups of  outer automorphisms of free groups. This gives a more palpable way to compare the deformation homology of the operad maps $E_n\to E_{n+k}$ for different~$k$ (essentially $k=0$ or 1 versus $k\geq 2$). 

\subsection*{Acknowledgements}
The authors thank Ryan Budney, Benoit Fresse, Allen Hatcher,  and Maxim Kontsevich for communication and helpful discussions. 

\section{Notation and prerequisites}
We generally work over a ground field $\K$ of characteristic zero, unless otherwise stated. 
For $V$ a graded or differential graded ($\K$-)vector space, we denote by $V[r]$ its $r$-fold desuspension.
We generally work in homological conventions, i.e., the differentials are generally of degree $-1$. Notice  that in some of the relevant references one uses the cohomological conventions, like for example in~\cite{grt}. One can easily switch from one setting to another by a grading reversion.
For the symmetric groups we use the notation $\bbS_n$. 
Regarding operads we mostly follow the conventions of the textbook by Loday and Vallette \cite{lodayval}.
In particular, $\Com$ and $\Lie$ are the commutative and Lie operads.
We denote by $\Poiss_n$ the $n$-Poisson operad generated by an abelian product operation $\wedge$ and a compatible Lie bracket $[,]$ of degree $n-1$. 
We furthermore denote by $\e_n$ the homology of the little $n$-cubes operad, without zero-ary operation. Concretely, $\e_1$ is the associative operad and $\e_n\cong\Poiss_n$ for $n\geq 2$.

For $\op P$ an operad and $r$ an integer we denote by $\op P\{r\}$ the operadic $r$-fold suspension.
For $\op P$ a quadratic operad we denote by $\op P^\vee$ its Koszul dual cooperad.
The most important example will be $\op P=\e_n$ with $\e_n^\vee = \e_n^*\{n\}$.

We denote by $\Omega(\op C)$ the cobar construction of a coaugmented cooperad $\op C$. Most importantly, we abbreviate $\hoe_n=\Omega(\e_n^\vee)$ and $\hoPoiss_n=\Omega(\Poiss_n^\vee)$. Furthermore, we will set $\hoLie_n:=\Omega(\Com^*\{n\})\subset \hoe_n$ to be the minimal resolution of the degree shifted Lie operad, so that one has a map $\hoLie_n\to \hoe_n$.

Let $f: \Omega(\op C)\to \op P$ be an operad map. Then we denote by 
\begin{equation}\label{equ:def_compl}
\Def(\Omega(\op C)\to \op P)= \Def(f) := \Hom_\bbS(\op C, \op P)^\alpha \cong\prod_N \Hom_{\bbS_N}(\op C(N), \op P(N))
\end{equation}
the operadic convolution dg Lie algebra twisted by the Maurer-Cartan element $\alpha$ describing the map $f$, cf. \cite[section 6.4.4]{lodayval}. Notice that this complex up to a shift in degree by one, almost coincides with the complex of derivations $\Der(f)$. The difference is that the complex of derivations possesses one extra homology class described as (arity - 1) rescaling of~$f$. This class in $\Def(f)$ is the boundary of ${\mathbf 1}\in \Hom (\op C(1),\op P(1))$, where $\mathbf 1$ is the composition $\op C(1)\to\K\to \op P(1)$ of counit and unit maps. We will denote by $\Der_*(f)$ the complex obtained from $\Der(f)$ quotiented out by this class. We will call it also {\it complex of reduced derivations}. To resume one has a quasi-isomorphism 
\[
 \Def(f) \simeq \Der_*(f)[1].
\]
More generally, given a morphism of dg-operads $\op P\to\op Q$ one can define a complex of its reduced homotopy derivations as the complex of reduced derivations of $\hat{\op P}\to\op Q$, where $\hat{\op P}$ is a cofibrant replacement of $\op P$:
\[
\hDer_*(\op P\to \op Q):=\Der_*(\hat{\op P}\to \op Q).
\]
By $\hDer_*(\op P)$ we will understand the homotopy derivation complex of the identity morphism $id\colon \op P\to \op P$. From a general theory it follows that complex $\hDer_*(\op P\to \op Q)$ carries a natural $\hoLie_2$-algebra structure~\cite{KS}, which is a generalization of the dg Lie algebra structure of~\eqref{equ:def_compl}. In all our examples cofibrant replacements appear as cobar constructions of cooperads. Thus for the most of the paper we will be using explicit complexes~\eqref{equ:def_compl} (except for  Section~\ref{s:emb_calcl} where it will be more convenient to use the grading conventions of the complex of reduced homotopy derivations). 

We will deviate slightly from the standard notion of grading and the associated graded for a filtration. We call a (complete) grading of a vector space $V$ a decomposition of $V$ into a direct product of subspaces
\[
 V\cong \prod_{i\in I} V_i.
\]
If $\mF$ is a complete descending filtration on a vector space $V$, we call it (complete) associated graded 
\[
 \gr V = \prod_{p} \mF^p / \mF^{p+1}.
\]
Finally, in this definition we may replace ``vector spaces'' in general by objects in some category. For example, by differential graded (in the usual sense) vector spaces.

Below, we will conduct several computations using complete filtrations spectral sequence arguments. The following (well-known) Lemma will suffice for our purposes.

\begin{lemma}\label{lem:SpecSeqConv}
Let $\mF^{\bullet}V$ be a descending complete bounded above filtration on a dg vector space $V$.\footnote{Concretely, completeness means that $V\cong \lim_{\leftarrow} V/\mF^p V$. Boundedness above means that for each $k$ there is an $N$ such that $\mF^{N}V^k = V^k$, where a superscript $k$ indicates taking the subspace of cohomological degree $k$. }
If $W\subset V$ is such that the induced map $H(\gr\, W)\to H(\gr V)$ is an isomorphism, then so is $H(W)\to H(V)$.
In particular if $H(\gr V)=0$ then $V$ is acyclic.
\end{lemma}
Note that boundedness above of the filtration is important.

\subsection{Cup product}\label{ss:cup_prod}
Let $\hoe_n \to \op P$ be an operad map, with $\op P$ an operad.
Then one can endow the desuspended deformation complex $\Def(\hoe_n\to \op P)[n]$ with a $\hoe_{n+1}$ algebra structure as shown by D. Tamarkin \cite{tamenaction}.
More concretely, there is an action of the higher braces operad $\Br_{n+1}$, which is a model for the $E_{n+1}$ operad \cite{CW}.
In particular, we may endow $\Def(\hoe_n\to \op P)[n]$ with a (homotopy commutative) product, the cup product. Concretely, one has the following explicit formulas.
We identify 
\[
\Def(\hoe_n\to \op P) \cong \Hom_\bbS(e_n^\vee, \op P)
\cong \prod_N (e_n\{-n\}(N) \otimes \op P(N))^{\bbS_N}
.
\]  
Suppose we are given two elements $x= \sum_j x_j'\otimes x_j''$ and $y=\sum_k y_k'\otimes y_k''$.
Furthermore denote the  Maurer-Cartan element corresponding to the above map $\hoe_n\to \op P$ by 
\[
m= \sum_l m_l' \otimes m_l''.
\]
Then, for $n\geq 2$
\begin{equation}\label{equ:cupproduct}
x\cup y = \sum_{j,k,l} \sum_\sigma \pm\sigma\cdot\left( \left(  (t_{12}\cdot m_l')\circ_{1,2}(x_j', y_k') \right) \otimes \left( m_l''\circ_{1,2}(x_j'', y_k'') \right) \right).
\end{equation}
Here $t_{12}\cdot$ denotes the operation on $e_n$ removing the edge between vertices $1$ and $2$ if there is one, and acting as zero if there is none. More formally, this is the coproduct in the Hopf operad $e_n$ followed by the projection of one factor onto the cogenerator which is the Lie bracket applied to inputs 1 and 2.
For $n=1$ one has that $e_1(N)\cong \K[\bbS_N]$ and one interprets $t_{12}\cdot m_l'$ as the projection (up to signs) to those terms for which the symbols $1$ and $2$ are in the correct order, i.e., 1 to the left of 2.
The second sum in \eqref{equ:cupproduct} is over shuffle permutations so as to symmetrize the result.

\section{Graph complexes and graph operads}
Let us briefly recall the construction of the Kontsevich graph complexes, and of the operads $\Graphs_n$, referring to \cite{grt} for more details.
We denote by $\gra_{N,k}$ the set of directed graphs with vertex set $[N]=\{1,\dots,N\}$ and edge set $k$.
It carries an action of the group $\bbS_N\times \bbS_k \ltimes \bbS_2^k$ by permuting the vertex and edge labels and changing the edge directions. The graphs operads $\Gra_n$ are defined such that
\[
\Gra_n(N) = \oplus_k (\K \langle \gra_{N,k}\rangle [(1-n)k])_{\bbS_k \ltimes \bbS_2^k}
\]
where the action of $\bbS_k$ is with sign if $n$ is even and the action of $\bbS_2^k$ is with sign if $n$ is odd.
Note that we allow loops (edges connecting a vertex to itself) in graphs in $\Gra_n$.\footnote{The notation thus deviates slightly from \cite{grt} where the symbol $\Gra_n^\circlearrowleft$ was used instead.}

The definition of $\Gra_n$ is made such that for all $n$ one has a map of operads
\begin{align*}
\Poiss_n &\to \Gra_n \\
\wedge &\mapsto 
\begin{tikzpicture}
\node[ext] at (0,0) {};
\node[ext] at (0.5,0) {};
\end{tikzpicture} 
\\
[,] &\mapsto 
\begin{tikzpicture}
\node[ext] (v) at (0,0) {};
\node[ext] (w) at (0.5,0) {};
\draw (v) edge (w);
\end{tikzpicture} .
\end{align*}
In particular, we obtain a map 
\[
\hoLie_n\to \hoPoiss_n\to \Poiss_n \to \Gra_n.
\]
The full graph complex is by definition the deformation dg Lie algebra
\[
\fGC_n := \Def(\hoLie_n\to \Gra_n).
\]
We will use two sub-complexes:
\begin{itemize}
\item The connected graphs with at least bivalent vertices form the sub-dg Lie algebra $\GC^2_n$.
\item The connected graphs with at least trivalent vertices form the sub-dg Lie algebra $\GC_n$.
\end{itemize}
One can check that (see \cite{grt})
\[
H(\GC_n^2)= H(\GC_n)\oplus \bigoplus_{1\leq r \equiv 2n-1 \text{ mod 4}} \K L_r
\]
where $L_r$ denotes the $r$-loop graph of degree $n-r$.
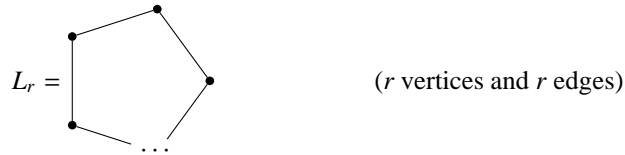
\[
L_r
=
\begin{tikzpicture}[baseline=-.65ex]
\node[int] (v1) at (0:1) {};
\node[int] (v2) at (72:1) {};
\node[int] (v3) at (144:1) {};
\node[int] (v4) at (216:1) {};
\node (v5) at (-72:1) {$\cdots$};
\draw (v1) edge (v2) edge (v5) (v3) edge (v2) edge (v4) (v4) edge (v5);
\end{tikzpicture}
\quad \quad \quad \quad \quad \quad \text{($r$ vertices and $r$ edges)}
\]

We may use the formalism of operadic twisting \cite{vastwisting} to twist the operad $\Gra_n$ to an operad $\fGraphs_n$. Elements of $\fGraphs_n(N)$ are series of graphs with two sorts of vertices, external vertices labelled $1,\dots,N$ and internal unlabeled vertices.
We again identify two useful suboperads
\begin{itemize}
\item The graphs with at least bivalent internal vertices and no connected components containing only internal vertices form the sub-operad $\Graphs^2_n$.
\item The graphs with at least trivalent internal vertices and no connected components containing only internal vertices form the sub-operad $\Graphs_n$.
\end{itemize}

The formalism of operadic twisting furthermore ensures that there is an action of the dg Lie algebra $\fGC_n$ on $\fGraphs_n$. One easily checks that the action restricts to an action of $\GC_n^2$ on $\Graphs_n^2$ and of $\GC_n$ on $\Graphs_n$.
Furthermore, the multiplicative group $\K^\times\ni \lambda$ acts on $\Graphs_n^2$ and $\Graphs_n$ by multiplying a graph $\Gamma$ by the number
\[
 \lambda^{\#(\text{internal vertices})-\#(\text{edges})}.
\]

There is a natural map $\Poiss_n\to \Graphs_n$ given by the same formulas as the map $\Poiss_n\to \Gra_n$ above. We will use the following well known result:
\begin{prop}[\cite{K2},\cite{LV},\cite{grt}]\label{p:quasi_iso}
The maps 
\[
\Poiss_n\to \Graphs_n^2 \to\Graphs_n 
\]
are quasi-isomorphisms.
\end{prop}

The composite map $\Poiss_n\to\Graphs_n$ is still an inclusion. This will allow us to view  elements of $\Poiss_n$  as (linear combinations of) graphs.

Finally, there is a natural (complete) topology on $\Graphs_n$ and $\Graphs_n^2$ induced by the filtration on the number of vertices.
It is clearly compatible with the operadic compositions, which are hence continuous.
Furthermore, there is a continuous Hopf operad structure on $\Graphs_n$ and $\Graphs_n^2$ as follows. 
We call a graph in $\Graphs_n$ or $\Graphs_n^2$ \emph{internally connected} if it is connected after deleting all external (numbered) vertices.
We denote the subspace of those graphs by $\ICG_n$ or $\ICG_n^2$.
Any general graph is obtained by gluing a (unique) set of internally connected graphs at the external vertices, and hence we may identify $\Graphs_n(N)\cong S(\ICG_n)$ and $\Graphs_n^2(N)\cong S(\ICG_n^2)$ with the completed free symmetric coalgebras. Thus one obtains the (complete) cocommutative coalgebra structure on each space $\Graphs_n(N)$.
The maps of Proposition \ref{p:quasi_iso} are compatible with the Hopf operad structures.

\subsection{The hairy graph complexes}\label{sec:HGC}
Let $\hoPoiss_m \stackrel{*}{\to} \Graphs_n$ be the composition
\[
 \hoPoiss_m\to \Poiss_m\to \Com \to \Graphs_n.
\]
The operadic deformation complex $\Def(\hoPoiss_m \stackrel{*}{\to} \Graphs_n)$ has been studied in \cite{Turchin3, Turchin2} and found to be quasi-isomorphic to the hairy graph complex $\fHGC_{m,n}$. Concretely, 
\begin{equation}
\label{equ:hairy_deform_comp}
\fHGC_{m,n}\subset \Def(\hoPoiss_m \stackrel{*}{\to} \Graphs_n) \cong \prod_{N\geq 1}\Hom_\bbS(\Poiss_m^\vee(N), \Graphs_n(N))
\end{equation}
is the dg Lie subalgebra of maps satisfying the following two conditions:
\begin{itemize}
 \item The map factors through the projection $\hoPoiss_m\to \hoLie_m$, i.e., all but the $\hoLie_m$ generators are sent to zero.
 \item Graphs in the image have all of their external vertices of valence one. The edges connecting to the external vertices we call the \emph{hairs} of the graph.
\end{itemize}

We define $\HGC_{m,n}\subset \fHGC_{m,n}$ as the subcomplex spanned by the connected graphs. The differential on $\fHGC_{m,n}$ leaves the number of connected components invariant and hence there is an isomorphism of complexes
\[
 \fHGC_{m,n} \cong S^+(\HGC_{m,n}[m])[-m]
\]
where $S^+$ denotes the completed symmetric algebra without constant term. Moreover the aforementioned isomorphism is an isomorphism of algebras for $m\geq 2$, where for the product on the left-hand side one takes the one induced by the cup product from Subsection~\ref{ss:cup_prod}. Indeed, the cup product preserves the subcomplex $\fHGC_{m,n}$ and graphically is described  as a disjoint union of graphs (in case of infinite sums of graphs one needs to distribute).

The differential in $\HGC_{m,n}$ preserves the first Betti number (number of loops) and the number of external vertices (number of hairs) in the graphs. The loopless part or in other words the tree part of $H(\HGC_{m,n})$ is always one dimensional and spanned by the graph
$$
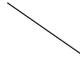
\begin{tikzpicture}
\draw (0,1) -- (1,.3);
\end{tikzpicture}
$$
in case of even codimension $n-m$ and by the {\it tripod} 
$$
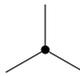
\begin{tikzpicture}[scale=.6]
\node[int] (v) at (0,0){};
\draw (v) -- +(90:1) (v) -- ++(210:1) (v) -- ++(-30:1);
\end{tikzpicture}
$$
in case of odd codimension. We will see that the tripod class is crucial for the codimension one relative non-formality.
The 1-loop part of $H(\HGC_{m,n})$ is generated by the hedgehog classes

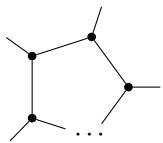
\[
\begin{tikzpicture}[baseline=-.65ex, scale=.7]
\node[int] (v1) at (0:1) {};
\node[int] (v2) at (72:1) {};
\node[int] (v3) at (144:1) {};
\node[int] (v4) at (216:1) {};
\node (v5) at (-72:1) {$\cdots$};
\draw (v1) edge (v2) edge (v5) (v3) edge (v2) edge (v4) (v4) edge (v5);
\draw (v1)  edge +(0:.6)  ; 
\draw (v2) edge +(72:.6)  ; 
\draw (v3) edge +(144:.6) ; 
\draw (v4)  edge +(226:.6)  ; 
\end{tikzpicture}
\]
that survive the dihedral symmetry, see~\cite[Proposition~3.3]{Turchin3}. The 2-loop part of $H(\HGC_{m,n})$ was computed in~\cite{Costello}.


\begin{ex}
 The following elements represent the simplest non-trivial classes in $H(\HGC_{1,2})$.
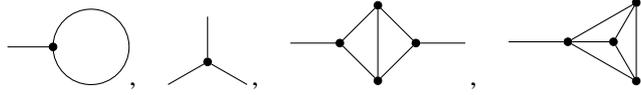
 \[
 \begin{tikzpicture}[scale=.5]
 \draw (0,0) circle (1);
 \draw (-180:1) node[int]{} -- +(-1.2,0);
 \end{tikzpicture}
,\quad
\begin{tikzpicture}[scale=.6]
\node[int] (v) at (0,0){};
\draw (v) -- +(90:1) (v) -- ++(210:1) (v) -- ++(-30:1);
\end{tikzpicture}
\,,\quad
\begin{tikzpicture}[scale=.5]
\node[int] (v1) at (-1,0){};\node[int] (v2) at (0,1){};\node[int] (v3) at (1,0){};\node[int] (v4) at (0,-1){};
\draw (v1)  edge (v2) edge (v4) -- +(-1.3,0) (v2) edge (v4) (v3) edge (v2) edge (v4) -- +(1.3,0);
\end{tikzpicture}
 \, ,\quad
 \begin{tikzpicture}[scale=.6]
\node[int] (v1) at (0,0){};\node[int] (v2) at (180:1){};\node[int] (v3) at (60:1){};\node[int] (v4) at (-60:1){};
\draw (v1) edge (v2) edge (v3) edge (v4) (v2)edge (v3) edge (v4)  -- +(180:1.3) (v3)edge (v4);
\end{tikzpicture}
 \, .
 \]
 The first class is the simplest one-loop class (hedgehog) with only one hair. It is responsible for the deformation of the commutative product in the direction of the bracket.  
\end{ex}

For more details on the complexes $\HGC_{m,n}$ we refer the reader to \cite{Turchin3}\footnote{Notice that our grading conventions for the hairy graph complexes $\HGC_{m,n}$, $\fHGC_{m,n}$ differ from those of~\cite{Turchin3} by an $m$-fold suspension.}.

Let us also comment on the combinatorial form of the Lie bracket on $\fHGC_{m,n}$, which arises by restricting the canonical Lie bracket of $\Def(\hoe_m \stackrel{*}{\to} \Graphs_n)$.
Given two hairy graphs $\Gamma, \Gamma'\in \fHGC_{m,n}$ of homogeneous degree, the Lie bracket is 
\[
 \Gamma \circ \Gamma' -(-1)^{|\Gamma||\Gamma'|} \Gamma' \circ \Gamma
\]
where $\circ$ is, roughly speaking, the operation of connecting one hair of $\Gamma$ to vertices of $\Gamma'$ in all possible ways. More precisely, if $\Gamma\in (\Graphs_n(N)[(m-1)N])^{\bbS_N}$ and $\Gamma'\in (\Graphs_{n}[(m-1)N])(M)^{\bbS_M}$ then
\begin{equation}\label{equ:circ}
\Gamma \circ \Gamma' = \sum_{\sigma\in \mathit{Sh}(N-1,M)} (-1)^{m(|\sigma|+N-1)} \Gamma(\sigma(1),\dots, \sigma(N-1), \Gamma'(\sigma(N),\dots, \sigma(N+M-1)))
\end{equation}
where the sum is over shuffle permutations and the notation means that one should relabel vertex $2$ of $\Gamma$ to $\sigma(1)$, etc. and insert $\Gamma'$ (with suitably relabeled vertices) into vertex $N$ of  $\Gamma'$.
Note that $\HGC_{m,n}\subset \fHGC_{m,n}$ is a Lie subalgebra.

Similarly to before we may also define the hairy graph complexes $\HGC_{m,n}^2\subset \fHGC_{m,n}^2\subset \Def(\hoe_m\to \Graphs_n^2)$ by allowing for bivalent (internal) vertices.


\subsection{Hodge grading/filtration}
The operads $\Poiss_m=\Com\circ \Lie\{m-1\}$ carry two natural gradings. The first one, that we will call {\it Hodge grading}, is by the number of iterated brackets used (equivalently the homological degree divided by $(m-1)$). The second grading, that we will call {\it dual Hodge grading} is  stemming by the arity on $\Com$ minus one.  For example, elements of $\Com(N)\subset \Poiss_m(N)$ have Hodge degree $0$ and dual Hodge degree $(N-1)$. The elements of $\Lie\{m-1\}(N)\subset \Poiss_m(N)$ have Hodge degree $(N-1)$ and dual Hodge degree~$0$. 
 These two gradings are Koszul dual to each other. To be precise if one considers the grading on $\hoPoiss_m=\Omega(\Poiss_m^*\{m\})$ induced by the dual Hodge grading on $\Poiss_m^*\{m\}$, then the differential on $\hoPoiss_m$ preserves this grading and moreover the natural quasi-isomorphism $\hoPoiss_m\to\Poiss_m$ respects this grading assuming that the target is endowed with the Hodge grading. For this reason, this grading on $\hoPoiss_m$ will be also called Hodge grading. 

The operads $\e_m$, for $m\geq 2$ are isomorphic to $\Poiss_m$ and hence inherit the Hodge grading.
The operad $\e_1$ is isomorphic to $\Poiss_1$ as an $\bbS$-module, but not as an operad.
Concretely, we fix the isomorphism $\Poiss_1\cong \e_1$ such that the following holds for any vector space $V$:
The map between the free associative (i.e., $\e_1$-) algebra on $V$ and the free Poisson ($\Poiss_1$-) algebra on $V$ induced by the identification of $\bbS$-modules $\Poiss_1\cong \e_1$ agrees with the Poincar\'e-Birkhoff-Witt map. Here one thinks of the free $\Poiss_1$ algebra as the symmetric algebra in a free Lie algebra on $V$ and of the free associative algebra on $V$ as its universal enveloping algebra.

The (dual) Hodge grading on $\Poiss_1$ may be transported to a filtration, the (dual) Hodge filtration, on $\e_1$ which turns out to be compatible with the operad structure. The Hodge filtration on $\e_1$ is descending  and the dual Hodge filtration is ascending. Both filtrations correspond to the Poincar\'e-Birkhoff-Witt filtration numbered differently. 
The associated graded operad of $\e_1$ is again the Poisson operad $\Poiss_1$,
\[
 \gr\, \e_1 \cong \Poiss_1.
\]
Concretely, we will need to use the Hodge filtration on $\hoe_1=\Omega(e_1^*\{1\})$ (induced by the dual Hodge filtration on $\e_1^*\{1\}$). Note that elements in the $p$-th subspace of the dual Hodge filtration on $e_1$ correspond to elements of dual Hodge degree $\leq p$ in $\Poiss_1$.

\subsection{Hodge grading/filtration in deformation homology. Hairy graphs and maps from $\e_1$}

The Hodge grading on $\hoPoiss_m$ induces an additional grading (that we will also call {\it Hodge}) on the deformation complex~\eqref{equ:hairy_deform_comp}. On the level of its hairy subcomplex $\fHGC_{m,n}$ it is simply the number of hairs minus one\footnote{This differs from the  definition in~\cite{Turchin1} by subtraction of one.}. From the geometrical description of the bracket and cup product on the hairy complex it is clear that the Lie bracket preserves the Hodge grading and the cup product decreases this grading by one.

Let $\hoe_1 \stackrel{*}{\to} \Graphs_n$ be the composition
\[
 \hoe_1\to \e_1\to \Com \to \Graphs_n.
\]
The operadic deformation complex $\Def(\hoe_1 \stackrel{*}{\to} \Graphs_n)$ has been studied in \cite{LambrechtsTurchin,Turchin1} and found to be quasi-isomorphic (as a complex) to the hairy graph complex $\fHGC_{1,n}$~\cite[Theorem~8.2]{Turchin1}.
Concretely, $\fHGC_{1,n}$ may be realized as a subcomplex of 
\[
 \Def(\hoe_1 \stackrel{*}{\to} \Graphs_n) \cong \prod_{N\geq 1}\Hom_\bbS(\e_1^\vee(N), \Graphs_n(N))\cong \prod_{N\geq 1}\Graphs_n(N)[N-1]
\]
as follows.
\begin{itemize}
\item Take the subcomplex of maps such that graphs in the image have all of their external vertices of valence one. 
 \item \sloppy Take a further subcomplex consisting of the graphs anti-invariant under the $\bbS_N$-actions on factors $\Graphs_n(N)$.
\end{itemize}

The construction is the same as that of section \ref{sec:HGC}.
The reason we consider it separately is that in contrast to before the map $\fHGC_{m,1}\to \Def(\hoe_1 \stackrel{*}{\to} \Graphs_n)$ now is no longer compatible with the Lie structure and the cup products.
The subcomplex $\fHGC_{m,1}$ is closed under neither of these operations.

However, the  Hodge filtration on $\hoe_1$ described in the previous section induces an ascending filtration on $\Def(\hoe_1 \stackrel{*}{\to} \Graphs_n)$ 
such that the associated graded can be identified with the deformation complex $\Def(\hoPoiss_1 \stackrel{*}{\to} \Graphs_n)$ from section \ref{sec:HGC}, into which $\fHGC_{1,n}$ injects compatibly with the Lie bracket and cup product.

\subsection{The filtration by the total excess}\label{sec:totalgenus}
Below we will study deformations of the operad maps $\hoe_m \stackrel{*}{\to} \Graphs_n$ that are not compatible with the Hodge filtration on $\Def(\hoe_m \stackrel{*}{\to} \Graphs_n)$ in some sense.
Therefore, we will introduce in this section a second related filtration, the \emph{total excess} filtration.

First consider the deformation complex
\[
\Def(\hoPoiss_m \stackrel{*}{\to} \Graphs_n)
\cong \prod_N \left( \Poiss_m\{-m\}(N) \otimes \Graphs_n(N) \right)^{\bbS_N}.
\]
We will prescribe a grading as follows. Let $x\otimes \Gamma\in \Poiss_m\{-m\}(N) \otimes \Graphs_n(N)$ be an element with $x\in \Poiss_m\{-m\}(N)$ homogeneous of dual Hodge degree $h$ and $\Gamma\in \Graphs_n(N)$ a graph with $k$ internal vertices and $\ell$ edges.
Then we assign to $x\otimes \Gamma$ the \emph{total excess}
\begin{equation}\label{equ:totalgenus}
 g = -h +\ell  - k.
\end{equation}
The terminology stems from the fact that thanks to Proposition~\ref{p:quasi_iso} one may regard elements of $\Poiss_m(N)$ of dual Hodge degree $h$ as certain graphs with $N$ vertices and $N-h-1$ edges. Then the total excess is one minus the Euler characteristic of the combined graph obtained by gluing the graph corresponding to $x$ to $\Gamma$ at the external vertices, or alternatively the number of edges exceeding those required for a spanning tree in the combined graph.

\begin{rem}
 Note that the total excess grading is  minus the sum of the grading on the operad $\Graphs_n$ by the ``internal Euler characteristic'' $k-\ell$, and the Hodge grading on $\Poiss_m$.
\end{rem}

\begin{lemma}\label{lem:tgen_compat1}
The total excess induces a (complete) grading on $\Def(\hoPoiss_m \stackrel{*}{\to} \Graphs_n)$, compatible with the dg Lie algebra structure and the cup product in the following sense:
\begin{itemize}
 \item The differential preserves the total excess.
 \item The Lie bracket of two homogeneous elements of total excess $g_1$ and $g_2$ is homogeneous of total excess $g_1+g_2$.
 \item The cup product of two homogeneous elements of total excess $g_1$ and $g_2$ is homogeneous of total excess $g_1+g_2-1$.
\end{itemize}
\end{lemma}
\begin{proof}
 The Lie bracket is functorial, i.e., it is built using only the operadic compositions on $\Poiss_m$ and on $\Graphs_n$. Hence any linear combination of any gradings on $\Poiss_m$ and $\Graphs_n$ is preserved by the Lie bracket. The differential is built from the differential on $\Graphs_n$, which clearly preserves the grading by definition of grading, and the Lie bracket with a Maurer-Cartan element. Hence if the Maurer-Cartan element is of (total excess-)degree $0$, so is the differential.
 
 Finally, the cup product is built from the Maurer-Cartan element using operadic compositions on $\Poiss_m$ and $\Graphs_n$ and the operation $t_{12}\cdot : \Poiss_m\to \Poiss_m$, cf. \eqref{equ:cupproduct}. Hence we are done if we can show that the map $t_{12}$ has dual Hodge degree $1$. But this is easily established as the operation removes one edge if we think of elements of $\Poiss_m$ as certain graphs.
\end{proof}

Using the identification $\e_m\cong\Poiss_m$ for $m\geq 2$ we hence obtain gradings by total excess on $\Def(\hoe_m \stackrel{*}{\to} \Graphs_n)$ compatible with Lie bracket and cup product. 

For $m=1$ we may still identify $\Def(\hoe_1 \stackrel{*}{\to} \Graphs_n)$ and $\Def(\hoPoiss_1 \stackrel{*}{\to} \Graphs_n)$ as dg vector spaces. However, since we do not have a Hodge grading, merely a Hodge filtration on the former space, we only obtain a \emph{filtration} by total excess on $\Def(\hoe_1 \stackrel{*}{\to} \Graphs_n)$. This filtration is descending as $h$ appears with negative sign in~\eqref{equ:totalgenus}.

\begin{lemma}\label{lem:tgen_compat2}
 The total excess filtration on $\Def(\hoe_1 \stackrel{*}{\to} \Graphs_n)$ is compatible with the Lie bracket and cup product in the following sense.
 \begin{itemize}
  \item The Lie bracket of two elements of total excess $\geq g_1$ and $\geq g_2$ is of total excess $\geq g_1+g_2$.
 \item The cup product of two homogeneous elements of total excess $\geq g_1$ and $\geq g_2$ is homogeneous of total excess $\geq g_1+g_2-1$.
\end{itemize}
\end{lemma}
\begin{proof}
The proof is a copy of that of Proposition \ref{lem:tgen_compat1}, except that one has to verify that the operation 
\[
 t_{12}\cdot \colon \e_1(N) \to \e_1(N)
\]
maps the subspace $\mF^p\e_1(N)$ of dual Hodge filtration $p$ to the subspace $\mF^{p+1}$ of dual Hodge filtration $p+1$. (Note that $t_{12}\cdot$ projects onto the subspace of permutations in $\e_1(N)\cong \K[\bbS_N]$ for which $1$ and $2$ occur in increasing oder.)
We leave the verification of this claim to the reader, using the definition of the Poincar\'e-Birkhoff-Witt map.
\end{proof}

\section{Review of M. Kontsevich's proof of the formality of the little $n$-cubes operads }\label{sec:formality review}
  Let us recall M. Kontsevich's proof of the formality of the little $n$-cubes operads, over $\K=\R$. The operad $E_1$ is obviously formal since the augmentation map $E_1\to e_1$ is  a quasi-isomorphism. Thus one can assume $n\geq 2$.
Instead of working with the little $n$-cubes operads directly, we will use another model, the compactified configuration spaces $\FM_n$. Concretely, $\FM_n(N)$ is a compactification of the configuration space of $N$ points in $\R^n$, modulo overall translation and rescaling.
We refer the reader to \cite{K2} for more details.
Let $C(\FM_n)$ be the operad of semi-algebraic chains on $\FM_n$. M. Kontsevich found the following zigzag of operad quasi-isomorphisms realizing the formality of the little $n$-cubes operads:
\[
 C(\FM_n) \to \Graphs_n \leftarrow \e_n.
\]
Here the first arrow is constructed as follows. A chain $c\in C(\FM_n(N))$ is sent to the series
\begin{equation}\label{equ:Kmap}
 \sum_\Gamma \Gamma \ \int_c \int_{\FM_{n}(N+k)/\FM_n(N)} \prod_{(ij)\in E\Gamma} \pi_{ij}^* \Omega_{S^{n-1}}
\end{equation}
where:
\begin{itemize}
 \item The sum is over graphs $\Gamma$ forming a basis of $\Graphs_n(N)$. The number $k$ is the number of internal vertices of $\Gamma$.
 \item The integral is along the fiber of the forgetful map $\FM_n(N+k)\to \FM_n(N)$.
 \item The product is over all edges of $\Gamma$.
 \item $\pi_{ij}:\FM_n(N+k)\to \FM_n(2)\cong S^{n-1}$ is the forgetful map forgetting all but the $i$-th and $j$-th points of a configuration. 
 \item Finally $\Omega_{S^{n-1}}$ is the round volume form on the $n-1$-sphere.
\end{itemize}

\subsection{Hopf cooperadic version of Kontsevich's morphism}
The operads $\Graphs_n$ are duals of Hopf cooperads $\stG_n$. Concretely, while elements of $\Graphs_n$ are series in isomorphism classes of certain graphs, elements of $\stG_n$ are just linear combinations of the same graphs, with the dual (adjoint) differential.
The commutative product on the $\stG_n(N)$ is defined by gluing graphs at the $N$ external vertices, and the unit is the graph without edges.

M. Kontsevich's construction above may be restated as providing a zig-zag
\[
 \Omega(\FM_n) \leftarrow \stG_n \to \e_n^*
\]
where $\Omega(\FM_n)$ denotes the $\bbS$-module of PA forms on $\FM_n$, cf. \cite{HLTV}.
Here the left-hand map $F:\stG_n\to \Omega(\FM_n)$ is defined such that 
\[
 F(\Gamma) = \int_{\FM_{n}(N+k)/\FM_n(N)} \prod_{(ij)\in E\Gamma} \pi_{ij}^* \Omega_{S^{n-1}}.
\]
It is not hard to check that $F$ preserves the commutative algebra structures.
We would like to say that $F$ is a map of Hopf cooperads. However, unfortunately $\Omega(\FM_n)$ is not a cooperad. Still, the map $F$ is compatible with the operad structure on $\FM_{n}$ in the sense that following diagrams commute.
\begin{equation}\label{equ:compat}
 \begin{tikzpicture}[baseline=-.65ex]
  \matrix[diagram](m) { \stG_n(N+M-1) & \Omega(\FM_n(N+M-1)) \\
  & \Omega(\FM_n(N)\times \FM_n(M)) \\
  \stG_n(N) \otimes \stG_n(M) &  \Omega(\FM_n(N)) \otimes \Omega(\FM_n(M)) \\ };
  \draw[-latex] (m-1-1) edge node[auto] {$F$} (m-1-2) edge node[left] {$\Delta_i$} (m-3-1)
                (m-3-1) edge node[auto] {$F\otimes F$} (m-3-2)
                (m-3-2) edge (m-2-2) 
                (m-1-2) edge node[auto] {$\circ_i^*$} (m-2-2);
 \end{tikzpicture}
\end{equation}
Here $\circ_i$ is the $i$-th operadic composition and $\Delta_i$ the $i$-th cooperadic cocomposition.

%
%
%
%
%
%

\section{Proof of Theorem \ref{thm:main} for $k\geq 2$} \label{sec:proof k geq 2}
In this section we will show Theorem \ref{thm:main} for $k\geq 2$. The proof is based on two Lemmas.
\begin{lemma}\label{lem:grvanishing}
Consider the composition 
\[
C(\FM_n) \to C(\FM_{n+k}) \to \Graphs_{n+k}.
\]
If $k\geq 2$ then the image is contained in the sub-operad $\Com$ of graphs without edges. (It is one-dimensional in each arity.) 
\end{lemma}
\begin{proof}
One has to check that for each graph $\Gamma\in \Graphs_{n+k}(r)$ with at least one edge the configuration space integrals 
\[
\int_c  \pi^* \omega_\Gamma
\]
vanish for each chain $c$ in $C(\FM_n(r))$. In fact we claim that $\pi^* \omega_\Gamma=0$.
It is sufficient to consider connected 
graphs for otherwise the differential form is just a product of pullbacks of those of its  connected components.

Roughly speaking the lowest degree of such form happens when $\Gamma$ is a uni-trivalent tree,  but let us still give a formal argument. Suppose that $\Gamma$ has $v$ internal vertices and $e$ edges. Since each internal vertex has to be at least trivalent and each external vertex has to be at least univalent, there is an inequality\footnote{In fact, since a blow up of a vertex in a graph decreases the degree,  one can assume without loss of generality that $\Gamma$ is uni-trivalent.}
\begin{equation}\label{equ:inequ1}
e \geq \frac 3 2 v + \frac 1 2 r.
\end{equation}
Now assume that the first Betti number of $\Gamma$ is $j$. By connectedness, one has
\begin{equation}\label{equ:b1}
e-v-r+1=j.
\end{equation}
Equivalently~\eqref{equ:b1} can be written as
\begin{equation}\label{equ:e}
e=v+r+j-1,
\end{equation}
hence
\begin{equation}\label{equ:inequ2}
v\leq r+2j-2.
\end{equation}
Applying \eqref{equ:e} and~\eqref{equ:inequ2} to the degree of the differential form $\omega_\Gamma$, we   obtain
\[
 \mathit{deg}(\omega_\Gamma)
=
(n+k-1)e - (n+k)v
=
(n+k-1)(r+j-1)-v
\geq 
n(r-1)+(n-1)j+(k-2)(r+j-1)+1.
\]
On the other hand, $\dim(\FM_n(r))=n(r-1)-1$. Hence for $k\geq 2$ we find that the restriction of $\omega_\Gamma$ on $\FM_n(r)$ is zero as claimed.

\end{proof}

We can use the above Lemma to show Theorem \ref{thm:main} for $k\geq 2$ as follows.
\begin{proof}[Proof of Theorem \ref{thm:main} for $k\geq 2$]
Consider the diagram
\begin{equation}\label{equ:formalitydiagram}
\begin{tikzpicture}[baseline=-.65ex]
\matrix[diagram](m) { C(\FM_n) & C(\FM_{n+k}) \\
\Graphs_n & \Graphs_{n+k} \\
\e_n & \e_{n+k} \\ };
\draw[-latex] (m-1-1) edge (m-1-2) edge (m-2-1)
		    (m-1-2) edge (m-2-2)
		    (m-2-1) edge (m-2-2)
		    (m-3-1) edge (m-3-2) edge (m-2-1)
		    (m-3-2) edge (m-2-2);
\end{tikzpicture}
\end{equation}
where the middle horizontal arrow sends all graphs with edges to zero, and the other arrows were introduced above. In case $n=1$, instead of $\Graphs_n$ in~\eqref{equ:formalitydiagram} one should use $\e_1$, and the upper left arrow is the augmentation morphism.
The lower square of the diagram obviously commutes.
The upper square commutes by Lemma \ref{lem:grvanishing}, as long as $k\geq 2$. Since all the vertical arrows are quasi-isomorphisms Theorem \ref{thm:main} hence follows in this case. 
\end{proof}

To attack the case $k=1$ we need additional results introduced in section \ref{sec:defcomplex}.

\subsection{Hopf (co)operadic version}
Note that the maps in the diagram \eqref{equ:formalitydiagram} may more or less obviously be pre-dualized to form a diagram
\[
\begin{tikzpicture}
\matrix[diagram](m) { \Omega(\FM_n) & \Omega(\FM_{n+k}) \\
\stG_n & \stG_{n+k} \\
\e_n^* & \e_{n+k}^* \\ };
\draw[latex-] (m-1-1) edge (m-1-2) edge (m-2-1)
		    (m-1-2) edge (m-2-2)
		    (m-2-1) edge (m-2-2)
		    (m-3-1) edge (m-3-2) edge (m-2-1)
		    (m-3-2) edge (m-2-2);
\end{tikzpicture}
\]
The maps in the lower square are Hopf cooperad maps. The upper three arrows are maps of $\bbS$-modules in commutative algebras, which in addition satisfy a compatibility condition with the operadic structure on $\FM_n$ and $\FM_{n+k}$ of the form \eqref{equ:compat}. Again, in case $n=1$, instead of $\stG_n$ in the diagram above  one should use $\e_1^*$.

\section{Homology of the deformation complex}\label{sec:defcomplex}

\subsection{An auxiliary map $\hoPoiss_{n-1}\to \Graphs_n^2$}
Consider the map $\Phi_0: \hoPoiss_{n-1}\to \Graphs_n^2$ defined as follows:
\begin{itemize}
 \item The commutative product generator $m_2\in \hoPoiss_{n-1}(2)$ is sent to the graph with two vertices and no edge.
 \item The $2k+1$-ary $\hoLie_n$ generator $\mu_{2k+1}$ is sent to the graph
 \[
  \frac 1 {4^{k}}\
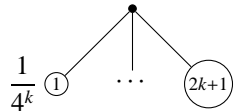
  \begin{tikzpicture}[baseline=-.65ex]
   \node[int] (v) at (0,1) {};
   \node[ext] (v1) at (-1,0) {$\scriptstyle 1$};
   \node (v2) at (0,0) {$\cdots$};
   \node[ext] (v3) at (1,0) {$\scriptstyle 2k+1$};
   \draw (v) edge (v1) edge (v2) edge (v3);
  \end{tikzpicture}
 \]
\item All other generators are sent to zero.
\end{itemize}


\begin{lemma}
 $\Phi_0$ is indeed a map of operads.
\end{lemma}
\begin{proof}
First note that the graphs in the images of each $\mu_{2k+1}$ are such that the external vertices all have valence 1, and are hence derivations in each slot for the product $m_2$.
The $\hoPoiss_{n-1}$-relations involving both the product $m_2$ and the $\mu_k$ are hence satisfied.
It suffices to check the $\hoLie_n$-relations.
These relations are equivalent to stating that the element 
\begin{equation}\label{equ:alphadef}
\alpha=
 \sum_{k\geq 1} 
 \frac 1 {4^k }
 \underbrace{
 \begin{tikzpicture}[baseline=-.65ex, scale=.5]
   \node[int] (v) at (0,1) {};
   \coordinate (v1) at (-1,0);
   \node (v2) at (0,0) {$\cdots$};
   \coordinate (v3) at (1,0);
   \draw (v) edge (v1) edge (v2) edge (v3);
  \end{tikzpicture}
  }_{2k+1\text{ legs}}
\end{equation}
is a Maurer-Cartan element in the hairy graph complex from section \ref{sec:HGC}.
The Maurer-Cartan equation reads
\begin{align*}
 -\delta \alpha +\frac 1 2 [\alpha,\alpha]&=-\delta + \alpha\circ \alpha \\
 &=
 \sum_{k,l} \left( -{k+l \choose k} + {k+l \choose k} \right)
 \begin{tikzpicture}[baseline=-.65ex, scale=.5]
  \node[int](v) at (0,1) {};
  \node[int](w) at (1,0) {};
  \draw (v) edge (w) edge +(-1.5,-1.5) edge +(0,-1.5) edge +(-.75,-1.5)
  (w) edge +(.5,-.5) edge +(0,-.5) edge +(1,-.5);
  \node at (-.75,-0.75) {$k\times$};
  \node at (1.5,-.75) {$l\times$};
 \end{tikzpicture}
 \\
 &=0.
\end{align*}
Here the first term in the sum is produced by $\delta$, there are ${k+l \choose k}$ ways of connecting the legs to one of the two vertices produced. The second term uses the pre-Lie product $\circ$ from \eqref{equ:circ}, and the prefactor counts the number of $(k,l)$-shuffle permutations.
\end{proof}

\begin{rem}
 The map $\Phi_0$ clearly factorizes through $\Graphs_n\subset \Graphs_n^2$. However, we will see below that allowing for bivalent internal vertices makes the theory more uniform.
\end{rem}

\subsection{The deformation complex of $\Phi_0$}
Consider the deformation complex $\Def(\Phi_0)$. It is obtained by twisting the deformation complex 
\[
 \Def(\hoPoiss_{n-1}\stackrel{*}{\to} \Graphs_n^2)
\]
by the Maurer-Cartan element $\alpha$ corresponding to $\Phi_0$. In fact, this Maurer-Cartan element lives in the sub-dg Lie algebra $\HGC_{n-1,n}\subset \fHGC_{n-1,n}\subset\Def(\hoPoiss_{n-1}\stackrel{*}{\to} \Graphs_n^2)$.
Hence we have a quasi-isomorphism of dg Lie algebras
\[
 \fHGC_{n-1,n}^\alpha \to \Def(\Phi_0)
\]
between the full hairy graph complex twisted by $\alpha$ and the deformation complex of $\Phi_0$.

The goal of this subsection is to show the following Theorem.
\begin{thm}\label{thm:Phi0}
\begin{equation}\label{equ:DefPhi0}
 H(\Def(\Phi_0)) 
 \cong S^+\left(\Bigl(H(\GC_n^2) \oplus \K T\Bigr)[n]\right)[1-n]
 \cong S^+\left(\Bigl(H(\GC_n) \oplus \prod_{1\leq r\equiv 2n+1 \text{ mod $4$}}\K L_r \oplus \K T\Bigr)[n]\right)[1-n]
\end{equation}
where the $L_r$ correspond to the loop classes and $T$ to the tripod class, and $S^+$ denotes the completed symmetric algebra without constant term.
\end{thm}

\begin{rem}
Note that the cup product on $\Def(\Phi_0)$ is combinatorially realized as the union of graphs.
\end{rem}

\begin{rem}\label{r:oddpod_class}
Concretely, a cocycle representing the tripod class $T$ is the following:
\begin{equation}\label{equ:Tgen}
 \sum_{k\geq 1} \frac{2k}{4^k}
 \underbrace{
\begin{tikzpicture}[baseline=-.65ex]
\node[int] (v) at (0,0) {};
\draw (v) edge +(-.5,-.5)  edge +(-.3,-.5) edge +(0,-.5) edge +(.3, -.5) edge +(.5,-.5);
\end{tikzpicture}
}_{2k+1}.
\end{equation}

The class corresponding to the cocycle $\gamma\in \GC_n^2$ is represented by a hairy graph cocycle as follows:
\begin{equation}\label{equ:gamma_img}
\gamma\mapsto \sum_{k\geq 0} \frac{1}{4^k} 
\sum 
\underbrace{
\begin{tikzpicture}[baseline=-.65ex]
\node (v) at (0,0) {$\gamma$};
\draw (v) edge +(-.5,-.5)  edge +(-.3,-.5) edge +(0,-.5) edge +(.3, -.5) edge +(.5,-.5);
\end{tikzpicture}
}_{2k+1}.
\end{equation}
Here the sum is over all ways of connecting the legs to vertices of $\gamma$.

It will be shown in Proposition \ref{prop:primaction} below that these cocycles corresponding to the generating classes arise from the action of the Lie algebra $\K \ltimes \GC_n^2$ on the operad $\Graphs_n^2$ by derivations. 
\end{rem}

\subsection{Proof of Theorem \ref{thm:Phi0}}
We will work with the hairy graph complex. Our eventual goal is to show that 
\[
H((\HGC_{n-1,n}^{2})^\alpha)\cong \K\oplus H(\GC_n^2)
\]
where $(\HGC_{n-1,n}^{2})^\alpha$ is the hairy graph complex with the differential obtained by twisting with the Maurer-Cartan element $\alpha$ defined in \eqref{equ:alphadef}.
First note that there is a splitting of complexes of $(\HGC_{n-1,n}^{2})^\alpha$ into a subcomplex of trees, and a subcomplex of non-trees, say
\[
 (\HGC_{n-1,n}^{2})^\alpha = U_{trees}\oplus U_{non-trees}.
\]
The subcomplex of trees has one-dimensional homology (being spanned by the tripod class) already before twisting by $\alpha$, see \cite[Subsection~3.5]{Turchin3}. Hence, the subcomplex of trees has one-dimensional homology and contributes the summand $H(U_{trees})\cong\K$ above. Remark~\ref{r:oddpod_class} explicitly describes the corresponding cocycle. We can hence focus on the subcomplex of non-tree graphs and disregard tree graphs below.

For technical reasons, let us enlarge the complex $\HGC_{n-1,n}^{2}$ to include also (non-tree) graphs without hairs, i.~e., set $\HGC_{n-1,n}^+ := \HGC_{n-1,n}^2\oplus \GC_n^2$. The dg Lie algebra  structure naturally extends to $\HGC_{n-1,n}^+$.
Our first goal will be to compute the homology of the twisted complex
\[
 (\HGC_{n-1,n}^+)^\alpha.
\]

Let $\Gamma\in \HGC_{n-1,n}^{+}$ be a graph.
Let the \emph{core} of $\Gamma$ be the graph obtained by removing the hairs and removing recursively univalent vertices and their adjacent edges. In other words, we may view $\Gamma$ as its core, with some trees attached, for example:
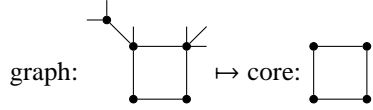
\[
 \text{graph: }
  \begin{tikzpicture}[baseline=-.65ex, scale=.35]
  \node[int](v) at (-1,-1) {};
  \node[int](w) at (-1,1) {};
  \node[int](a) at (1,-1) {};
  \node[int](b) at (1,1) {};
  \node[int](c) at (-2,2) {};
  \draw (v) edge (w) edge (a) (b) edge (a) edge (w) (w) edge (c) edge +(0,.75) (b) edge +(0,.75) edge +(.75,0) edge +(.75,.75) (c) edge +(0,.75) edge +(-.75,0);
 \end{tikzpicture}
 \mapsto
 \text{core: }
 \begin{tikzpicture}[baseline=-.65ex, scale=.35]
  \node[int](v) at (-1,-1) {};
  \node[int](w) at (-1,1) {};
  \node[int](a) at (1,-1) {};
  \node[int](b) at (1,1) {};
  \draw (v) edge (w) edge (a) (b) edge (a) edge (w);
 \end{tikzpicture}
\]

We endow $\HGC_{n-1,n}^+$ with a filtration on the number of core vertices.
The first differential in the associated spectral sequence will leave the number of core vertices invariant, but changes the attached trees.
The associated graded complex hence splits
\[
 gr\, \HGC_{n-1,n}^+\cong \prod_{\gamma} V_\gamma
\]
where the direct product runs over all isomorphism classes of core graphs and $V_\gamma$ is the complex associated to the core.
More precisely, we may pick a representative graph of the isomorphism class, say $\gamma'$, and identify
\[
 V_\gamma \cong \left(\bigoplus_{v\in V\gamma'} W \right)^{G_{\gamma'}}
\]
where the direct sum runs over vertices of $\gamma'$, the complex $W$ is made of the trees attachable to one vertex and the isomorphism group $G_{\gamma'}$ of $\gamma'$ acts by permutations, with appropriate signs.
Since taking homology interchanges with taking invariants with respect to a finite group action, it suffices to compute $H(W)$.
\begin{lemma}
 $H(W)\cong \K e_+\oplus \K e_-$ where the two classes correspond to decorations of the vertex of the form
 \begin{align*}
  e_\pm = \sum_{j\geq 0} \left(\frac{\pm 1}{2}\right)^j 
  \underbrace{
  \begin{tikzpicture}
   \node[int, label={$v$}] (v) at (0,0) {};
   \draw (v) edge +(-.5,-.5) edge +(-.33,-.5)  edge +(.33,-.5) edge +(.5,-.5);
  \end{tikzpicture}
  }_{j\times}.
 \end{align*}
\end{lemma}
\begin{proof}[Proof sketch]
 First one checks that $e_\pm$ are indeed cocycles. We leave it to the reader.\footnote{Hint: check that $e_++e_-$ and $e_+-e_-$ are closed.}
 Then consider a filtration on $W$ by the number of hairs in the decoration.
 The associated graded decomposes as graded vector space as
 \[
  gr\, W = W_0 \oplus W_1 \oplus W_{\geq 2}
 \]
where the subscript indicated the valence of the vertex $v$.
We take the associated spectral sequence whose first differential is a map $W_{\geq 2}\to W_1$. This differential is injective, and the cokernel is spanned by the decoration consisting of a single hair.
Hence on this page we arrive at a two-dimensional vector space, and since the two classes can be represented by cocycles in $W$ the spectral sequence abuts.
\end{proof}

Knowing the homology of $W$, we hence conclude that $H(gr\,\HGC_{n-1,n}^+)$ is spanned by isomorphism classes of graphs without hair, but with each vertex decorated by either a symbol $e_+$ or a symbol $e_-$.
In fact, interpreting $e_\pm$ to be the corresponding decorations above, such graphs form a subcomplex which we denote by 
\[
 \GC_n^\pm \subset \HGC_{n-1,n}^+.
\]
Using Lemma \ref{lem:SpecSeqConv} and the fact that the filtration used is bounded above and complete, we have hence reduced the problem to computing the homology of $\GC_n^\pm$.

Note that the differential on $\GC_n^\pm$ is such that it splits a vertex decorated by $e_+$ into two vertices decorated by $e_+$ and a vertex decorated by $e_-$ into two vertices decorated by $e_-$. Pictorially:
\begin{align*}
 \begin{tikzpicture}[baseline=-.65ex]
  \node[int, label=90:{$e_+$}] (v) at (0,0) {};
  \draw (v) edge +(-.35,0) edge +(-.35, -.35) edge +(-.35, .35);
   \draw (v) edge +(.35,0) edge +(.35, -.35) edge +(.35, .35);
 \end{tikzpicture}
&\mapsto\sum \ 
 \begin{tikzpicture}[baseline=-.65ex]
  \node[int, label=90:{$e_+$}] (v) at (0,0) {};
  \node[int, label=90:{$e_+$}] (w) at (0.5,0) {};
  \draw (v) edge (w);
  \draw (v) edge +(-.35,0) edge +(-.35, -.35) edge +(-.35, .35);
   \draw (w) edge +(.35,0) edge +(.35, -.35) edge +(.35, .35);
 \end{tikzpicture}
 &
  \begin{tikzpicture}[baseline=-.65ex]
  \node[int, label=90:{$e_-$}] (v) at (0,0) {};
  \draw (v) edge +(-.35,0) edge +(-.35, -.35) edge +(-.35, .35);
   \draw (v) edge +(.35,0) edge +(.35, -.35) edge +(.35, .35);
 \end{tikzpicture}
&\mapsto\sum \ 
 \begin{tikzpicture}[baseline=-.65ex]
  \node[int, label=90:{$e_-$}] (v) at (0,0) {};
  \node[int, label=90:{$e_-$}] (w) at (0.5,0) {};
  \draw (v) edge (w);
  \draw (v) edge +(-.35,0) edge +(-.35, -.35) edge +(-.35, .35);
   \draw (w) edge +(.35,0) edge +(.35, -.35) edge +(.35, .35);
 \end{tikzpicture}
\end{align*}

Let $\Gamma\in \GC_n^\pm$ be a graph. We call its \emph{skeleton} the graph obtained by iteratively removing all bivalent vertices and connecting their adjacent edges, until an at least trivalent graph or a graph with only one vertex is produced. Consider only the at least trivalent case for the moment.
The graph $\Gamma$ can be seen as the core, together with a ``decoration'' at each edge, the decoration being a string of bivalent vertices, each labelled by $e_+$ or $e_-$.

Consider a filtration by the number of skeleton vertices. The differential on the associated graded acts indepently on the decorations of the edges. Hence it suffices to compute the homology of the complex of decorations associated to one edge. We have to consider two cases.
\begin{itemize}
 \item The endpoints of the edge have the same labels, i.e., both $e_+$ or both $e_-$.
 \item The endpoints of the edge have the opposite labels, i.e., one endpoint is labelled $e_+$ and one $e_-$.
\end{itemize}
We leave it to the reader to check that in the second case the corresponding complex is acyclic, while in the first case the homology is one-dimensional, the class being represented by a direct edge between the endpoints.
A similar argument also shows that if the skeleton has a single bivalent vertex the same conclusion holds.

Note that if a connected graph does not have all vertex labels alike (i.e., all $e_+$ or all $e_-$), then it necessarily contains an edge between vertices with opposite labels.
Hence we conclude that the homology of the associated graded of $\GC_n^\pm$ consists of two copies of $\GC_n^2$, one embedded by giving all vertices labels $e_+$, and one by giving all vertices labels $e_-$.

Invoking Lemma \ref{lem:SpecSeqConv} again this shows that there is a quasi-isomorphism of complexes
\[
 \GC_n^2\oplus \GC_n^2 \to \GC_n^\pm
\]
and hence a quasi-isomorphism of complexes 
\[
 \K \oplus \GC_n^2\oplus \GC_n^2 \to \HGC_{n-1,n}^+.
\]
Hence to reach our goal of computing $H(\HGC_{n-1,n}^2)$ it suffices to express this space in terms of the homology of $\HGC_{n-1,n}^+$. To this end, consider the short exact sequence
\[
 0\to \HGC_{n-1,n}^2 \to \HGC_{n-1,n}^+ \to \GC_n^2\to 0.
\]
Note also that the induced map in homology $H(\HGC_{n-1,n}^+)\cong \K\oplus H(\GC_n^2)\oplus H(\GC_n^2) \to H(\GC_n^2)$ is obtained from the identity map on each of the $\GC_n^2$-summands and is hence surjective. Thus, we conclude that $H(\HGC_{n-1,n}^2)$ is the kernel of the above map, and can be identified with 
\[
 \K \oplus H(\GC_n^2).
\]
Concretely, the identification may be realized by a quasi-isomorphism
\[
 \K \oplus \GC_n^2 \to \HGC_{n-1,n}^2
\]
which sends a graph $\Gamma\in \GC_n^2$ to $\Gamma_+-\Gamma_-$, where $\Gamma_\pm$ is the graph with all vertices decorated by $e_+$, respectively by $e_-$. This then concludes the proof of Theorem \ref{thm:Phi0}.
\hfill\qed

\subsection{The primitive cocycles originate from derivations of the target.}
Above we have computed the homology of $\Def(\Phi_0)$ in terms of the graphs homology $H(\GC_n^2)$. Although explicit combinatorial formulas for the cocycles representing the homology classes were given, let us trace their origin by showing that the primitive elements are the images of the derivations of the target operad $\Graphs_n^2$ given by $\K S\ltimes \GC_n^2$.

The natural action of $\K S \ltimes \GC_n^2$ on $\Graphs_n^2$ is realized as follows.
$\K S$ (scaling) denotes one class which acts on elements $\gamma\in \GC_n^2$ by multiplication by the Euler characteristic and on $\Gamma\in \Graphs_n^2$ by the number
\[
\#(\text{internal vertices}) - \#(\text{internal edges}).
\]  

The action of a graph $\gamma\in \GC_n^2$ on $\Gamma\in\Graphs_n^2(N)$ is given by the formula  
\[
\gamma \cdot \Gamma = \pm \Gamma \bullet \gamma \pm\gamma_1\circ_1 \Gamma +\sum_{j=1}^N \pm\Gamma \circ_j \gamma_1.
\]
Here $\gamma_1\in \Graphs_n^2(1)$ is obtained by declaring the first vertex of $\gamma$ external and $\Gamma\bullet\gamma$ denotes the operation of inserting $\gamma$ into all internal vertices of $\Gamma$.
The sum over $j$ we will also abbreviate as $\Gamma\circ \gamma_1$.

Hence for any operad map $\op P\to \Graphs_n^2$ we obtain a map 
\[
\K S \ltimes H(\GC_n^2) \to H(\Def(\op P\to \Graphs_n^2))[-1].
\]

\begin{prop}\label{prop:primaction}
The image of $\K S \ltimes H(\GC_n^2)$ in  $H(\Def(\hoPoiss_{n-1}\stackrel{\Phi_0}{\longrightarrow} \Graphs_n^2))[-1]$ may be identified with the primitive part of \eqref{equ:DefPhi0}. Furthermore, the images of the loop classes have representatives (up to prefactors)
\[
H_r
\propto
\sum_{k_1,\dots, k_r\geq 1 \text{ odd}}
(const)
\begin{tikzpicture}[baseline=-.65ex]
\node[int] (v1) at (0:1) {};
\node[int] (v2) at (72:1) {};
\node[int] (v3) at (144:1) {};
\node[int] (v4) at (216:1) {};
\node (v5) at (-72:1) {$\cdots$};
\draw (v1) edge (v2) edge (v5) (v3) edge (v2) edge (v4) (v4) edge (v5);
\draw (v1) edge +(-30:.3) edge +(0:.4) edge +(30:.3) +(0:.5) node  {$k_1$}; 
\draw (v2) edge +(42:.3) edge +(72:.4) edge +(102:.3) +(72:.5) node {$k_2$}; 
\draw (v3) edge +(114:.3) edge +(144:.4) edge +(184:.3) +(144:.5) node {$k_3$}; 
\draw (v4) edge +(196:.3) edge +(226:.4) edge +(256:.3) +(-144:.5) node {$k_4$}; 
\end{tikzpicture}
\]
and hence correspond to the hedgehog classes of the untwisted hairy graph homology $H(\HGC_{n-1,n})$.
\end{prop}
\begin{proof}
First consider the rescaling class $S$. By straightforward computation its action on the Maurer-Cartan element yields the cocycle \eqref{equ:Tgen} starting with the tripod.

Next, the image of a graph homology class represented by the cocycle $\gamma\in \GC_n$ is represented by 
\[
\pm m\bullet \gamma \pm [m, \gamma_1]
\]
where $m$ is the Maurer-Cartan element corresponding to $\Phi_0$. Adding the coboundary of $\gamma_1$, i.e., $\delta\gamma_1+[m,\gamma_1]$, where $\delta$ is the differential on $\Graphs$  we obtain
\[
\pm \delta \gamma_1 \pm m\bullet \gamma.
\]
Note that $\delta \gamma_1$ is an element obtained by attaching one edge with a univalent external vertex to $\gamma_1$ in all possible ways by closeness of $\gamma$.
It follows that for any graph homology class the formula \eqref{equ:gamma_img} describes its image in the homology of the deformation complex.
It remains to check that for $\gamma$ the loop class $L_r$ the image may indeed be represented by a cocycle with the hedgehog leading term as shown above.
From the proof of Theorem \ref{thm:Phi0} one sees that the images of the $L_r$ span the one-loop part of the homology of $\HGC_{n-1, n}^\alpha$, which is one-dimensional in each degree. Hence it suffices to check that the above cocycles do so as well. Take a filtration on $\HGC_{n-1, n}^\alpha$ by the number of hairs. Then the associated graded is isomorphic to the untwisted complex $\HGC_{n-1, n}$. The one-loop part of its homology is already one-dimensional in each degree, spanned by the hedgehog classes, and hence the one-loop part of the homology of $\HGC_{n-1, n}^\alpha$ is spanned by (any) set of classes whose lowest-number-of-hairs terms are the hedgehog diagrams.

\end{proof}

%
%
%

\subsection{The map $\Phi$}
We consider next the map $\Phi:\hoe_{n-1}\to \Graphs_n^2$ defined as the composition 
\[
\Phi: \hoe_{n-1}\to C(\FM_{n-1}) \to C(\FM_{n}) \to \Graphs_n\to \Graphs_n^2
\]
where the left hand arrow is some homotopy lift of the zigzag (with $m=n-1$) 
\[
\hoe_{m}\to \e_{m} \to \Graphs_{m}  \leftarrow C(\FM_{m})
\]
and the next to the last arrow is the Kontsevich integration map.

\begin{lemma}
The map $\hoe_{m}\to C(\FM_{m})$ sends the $L_\infty$ generator $\mu_r$ to the fundamental chain $\Fund_r\in C(\FM_{m}(r))$.
\end{lemma}
\begin{proof}
The statement is shown by induction on $r$. Since the map $\hoe_{m}\to C(\FM_{m})$ is required to be the identity on homology, the statement holds for $r=2$. Suppose it holds up to some $r$, and suppose that $\mu_{r+1}$ is sent to the chain $c$. Then $c$ has to satisfy
\[
 \partial c = \partial \Fund_{r+1}
\]
or in other words $\Fund_{r+1}-c$ is a cycle. But since there is no homology in that degree in $\FM_{n}$ it must be a boundary, i. e., $\Fund_{r+1}-c= \partial x$ for some chain $x\in C(\FM_m(r+1))$ of degree $mr$. But $\mathit{dim}(\FM_m(r+1))=mr-1$, hence $x$ must be degenerate, and hence $x=0$ since there are no such semi-algebraic chains by definition.
\end{proof}

\begin{rem}\label{r:lift_m=1}
For the case $m=1$ we can choose this lift $\hoe_1\to C(\FM_1)$ explicitly by sending the generator $m_r$ of $\hoe_1=A_\infty$ to the fundamental class of the connected component of $\FM_1(r)$ corresponding to the trivial permutation of $1\ldots r$.
\end{rem}

We will first show that the map $\Phi$ is a deformation of $\Phi_0$, in the sense that $\Phi_0$ is the associated graded map of $\Phi$ when considering a suitable filtration on $\Graphs_{n}$.
\begin{lemma}\label{lem:fundimage}
The image of the fundamental chain $\Fund_3\in C(\FM_{n-1}(3))$ under the composition 
\[
C(\FM_{n-1})\to C(\FM_{n}) \to \Graphs_{n}
\]
contains the tripod graph 
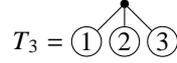
\[
T_3 =
\begin{tikzpicture}[baseline=-.65ex]
\node[ext] (v1) at (-.5,0) {$1$};
\node[ext] (v2) at (0,0) {$2$};
\node[ext] (v3) at (.5,0) {$3$};
\node[int] (v) at (0,.5) {};
\draw (v) edge (v1) edge (v2) edge (v3); 
\end{tikzpicture}
\]
with coefficient $\frac 1 4$, i.e., 
\[
\Fund_3 \mapsto \frac 1 4 T_3 + (\text{graphs with more vertices}).
\]
\end{lemma}
\begin{proof}
Consider the forgetful map $\pi : \FM_{n}(4)\to \FM_{n}(3)$, and consider $\Fund_3$ as an element of $C(\FM_{n}(3))$ via the inclusion $C(\FM_{n-1})\to C(\FM_{n})$. We have to compute the integral
\[
\int_{\pi^* \Fund_3} \tilde \omega_{T_3}.
\]
Note that $\FM_{n}$ carries an orientation reversing involution $I$ by mirroring along the $x_1,...,x_{n-1}$-plane. This involution leaves $\Fund_3$ invariant. 
We will decompose $\pi^* \Fund_3=:c+I (c)$ where the chain $c$ is supported in the subspace of $\FM_{n}(4)$ for which the 4-th point lies above the other three.\footnote{More precisely, one has projections $\pi_{i}:\FM_{n}(4)\to \FM_{n}(2)\cong S^{n-1}$ by forgetting all but the $i$-th and the fourth point of a configuration. Then we require that $c$ is supported in $\pi_{1}^{-1}(U)\cap \pi_{2}^{-1}(U)\cap \pi_{3}^{-1}(U)$, where $U\subset S^{n-1}$ is the closed upper hemisphere.}
One then finds 
\[
\int_{\pi^* \Fund_3} \tilde \omega_{T_3} = 2\int_{c} \tilde \omega_{T_3}
\]
since 
\[
\int_{I(c)}  \tilde \omega_{T_3} =
- \int_{c}  I^*(\tilde \omega_{T_3})
= \int_{c}  \tilde \omega_{T_3}.
\]
Here the first minus sign arises since $I$ is orientation reversing and the second since $I$ acts by multiplication with $-1$ on each of the three one-forms associated to the edges.
It remains to compute the integral $\int_{c} \tilde \omega_{T_3}$.
This can be done by parameterizing configurations $(X_1,X_2,X_3,X_4)$ as follows
\begin{align*}
X_j&=(x_j,-1), \quad j=1,2,3 & X_4=(0,0,0,0)
\end{align*}
with $x_1,x_2,x_3 \in \R^{n-1}$. Then integral then reduces to 
\[
\left( \int_{x\in \R^{n-1}} \Omega(x) \right)^3 = \left(\frac 1 2 \right)^3= \frac 1 8
\]
where $\Omega$ is the top form obtained by pulling back the volume form on the sphere $S^{n-1}$.
\end{proof}
\begin{rem}
By a similar argument one may also show that the the coefficient of the $r$-pod diagram in the image of the fundamental chain of $\FM_{n-1}(r)$ has coefficient $\frac 1 {2^{r-1}}$ for $r$ odd and 0 for $r$ even.
\end{rem}


\begin{lemma}\label{lem:excess_filtr}
 The differential on $\Def(\hoe_{n-1}\stackrel{\Phi}{\longrightarrow}\Graphs_n^2)$ is compatible with the total excess filtration of section~\ref{sec:totalgenus}.
 The associated graded complex can be identified with $\Def(\hoPoiss_{n-1}\stackrel{\Phi_0}{\longrightarrow}\Graphs_n^2)$.
\end{lemma}
\begin{proof}
The deformation complex $\Def(\Phi)$ is obtained from $\Def(\hoe_{n-1}\stackrel{*}{\to}\Graphs_n^2)$ by twisting with some Maurer-Cartan element, say $m$, while $\Def(\Phi_0)$ is obtained by twisting with a Maurer-Cartan element $m_0$.
Hence we have to check that $m$ consists of terms of total excess $\geq 0$, and that $m$ and $m_0$ agree modulo terms of total excess~$\geq 1$.

To this end consider an element $x\otimes \Gamma\in \e_{n-1}\{1-n\}(N)\otimes \Graphs_n^2(N)$, with $x$ of dual Hodge degree $h$ (filtration $h$ in case $n=2$) and $\Gamma$ a graph with $k$ internal vertices and $\ell$ edges.
Note that if the graph $\Gamma\in \Graphs_{n}(N)$ has edges between external vertices the term vanishes, since those edges correspond to vanishing forms under the Feynman rules \eqref{equ:Kmap}.
It follows that possible $\Gamma$ occurring in $m$ must be such that either the number of internal vertices is $k\geq 1$, or $\Gamma$ is the ``empty'' graph with no edges, i.e., $\ell=k=0$. 

Using \eqref{equ:totalgenus} the total degree of our element $x\otimes \Gamma$ is
\begin{equation}\label{equ:excessformula}
 (degree) = (n-2)p+(n-1)\ell -(n-1)(N-1) - nk = (n-1)g - k - p
\end{equation}
where we set $p:=N-h-1$, and $g$ is the total excess. Since the degree of $m$ is $-1$, there are the following cases to be considered:
\begin{itemize}
\item Note that the terms with $\Gamma$ trivial vanish by degree reasons for $N\neq 2$, and for $N=2$ they produce the map $\hoe_{n-1}\stackrel{*}{\to}\Graphs_n^2$ factoring through $\Com$.
Hence we can disregard terms with trivial $\Gamma$ in the following.
 \item If $n\geq 3$ necessarily $g\geq 0$, and if $g=0$ than either $p=0$ and $k=1$ or $p=1$ and $k=0$. In the former case ($p=0$) the element $x\in \e_{n-1}\{1-n\}(N)$ is in the commutative suboperad and the contributing terms have been evaluated in Lemma \ref{lem:fundimage} and the remark following it. In the latter case ($k=0$) the graph $\Gamma$ cannot have any edges as we argued above, and hence we can disregard it.
 \item If $n=2$ either $g=0$ or $g=-1$. If $g=0$ the same arguments as for $n\geq 3$ show that the possible terms are those of $m_0$. If $g=-1$ then necessarily $k=p=\ell=0$, i.e., $\Gamma$ is trivial and can be disregarded.
\end{itemize}

\end{proof}

\begin{rem}\label{rem:excessboundedabove}
Note that from \eqref{equ:excessformula} it follows that 
\[
g = \frac 1 {n-1}( k+p+(degree) ) \geq \frac 1 {n-1} (degree)
\]
and hence that the excess filtration is bounded above.
\end{rem}
%
%
%

\begin{lemma}
 The cup product on $\Def(\Phi)$ is compatible with the total excess filtration in the sense that the cup products of two elements of total excess $\geq p$ and $\geq q$ has total excess $\geq p+q-1$. The induced cup product on the associated graded agrees with the cup product on $\Def(\Phi_0)$.
\end{lemma}
\begin{proof}
Noting that the Maurer-Cartan element is of total excess $\geq 0$, with leading part equal to $m_0$, the proof is the same as that of Lemmas \ref{lem:tgen_compat1} and \ref{lem:tgen_compat2}.
\end{proof}

We will also need the following property of $\Phi$:

\begin{lemma}\label{lem:Phi_m2}
The map $\Phi\colon \hoe_{n-1}\to\Graphs_n^2$ always sends the product generator to the graph with two vertices and no edge.
\end{lemma}
\begin{proof}
The case $n=2$ is a direct consequence of the Kontsevich vanishing lemma~\cite[Theorem~6.5]{K1}, see also~\cite{Khovanskii}. The case $n\geq 3$ is proved by the same argument as Lemma~\ref{lem:grvanishing} by counting the degree of the  forms corresponding to graphs with two external vertices.
\end{proof}

\subsection{A version of the Cerf Lemma}

\begin{thm}[Algebraic version of the Cerf Lemma]\label{thm:PhiDef}
\begin{multline*}
 H(\Def(\Phi))\cong H(\Def(\Phi_0)) 
 \cong S^+\left(\Bigl(H(\GC_n^2) \oplus \K T\Bigr)[n]\right)[1-n]
 \cong\\ S^+\left(\Bigl(H(\GC_n) \oplus \prod_{1\leq r\equiv 2n+1 \text{ mod $4$}}\K H_r \oplus \K T\Bigr) [n]\right)[1-n]
\end{multline*}
where $S^+$ denotes the completed symmetric algebra without constant term, the $H_r$ denote the hedgehog classes and $T$ the tripod class.
\end{thm}
\begin{proof}
We consider a spectral sequence induced by the genus filtration. It follows from Lemma~\ref{lem:excess_filtr} that 
on the first page we find $H(\Def(\Phi_0))$.
We claim that the spectral sequence abuts there. 
Indeed, note that all primitive classes stem from an action on $\Graphs_n^2$ and hence are represented by cocycles in $\Def(\hoe_n\to \Graphs_n^2)\cong \Def(\Phi)$. 
Representatives for the non-primitive classes can be obtained by taking (series of) cup products of the representatives of the primitve classes. Invoking again Lemma \ref{lem:SpecSeqConv} (which is applicable by Remark \ref{rem:excessboundedabove}) the result follows.
\end{proof}

\begin{rem}
Note that the above Theorem says that all homology in $H(\Def(\Phi))$ is generated by the image of the derivations of the target $H(\hDer_*(E_n))$.
It does not say however that the map $H(\hDer_*(E_n))\to H(\Def(\Phi))[-1]$ is an isomorphism.
Rather, both homology spaces $H(\hDer_*(E_n))$ and $H(\Def(\Phi))$ are complete symmetric algebras, with products of different degrees, and the map is an isomorphism on the primitive elements.\footnote{The deformation homology $H(\hDer_*(E_n))$ was computed in~\cite[Theorem~1.3]{grt}:
\[ H(\hDer_*(E_n)) \cong S^+\left(\Bigl(H(\GC_n^2) \oplus \K T\Bigr)[n+1]\right)[-n-1].\]}
We conjecture that a nicer version of the above algebraic Cerf Lemma can be formulated when considering $E_{n-1}$ and $E_n$ as Hopf operads, and the Hopf operadic deformation complexes. (This is a natural setting from the rational homotopy point of view~\cite{Fresse}.)
We expect that the homology of the Hopf operadic deformation complexes is just the primitive part of $H(\hDer_*(E_n))$ and $H(\Def(\Phi))$, and the map hence becomes an isomorphism in the Hopf setting.
\end{rem}

\section{Proof of Theorem \ref{thm:main} for $k=1$}

There are different ways to show that the operad maps $C(\FM_n)\to C(\FM_{n+1})$ are non-formal.  In Subsection~\ref{ss:alternative} we resume briefly two other approaches.  The simplest argument that we found is as follows.  We consider the two maps $\Phi:\hoe_{n}\to \Graphs_{n+1}^2$ and $\Psi:\hoe_n\to e_n\stackrel{*}{\to} e_{n+1}\to \Graphs_{n+1}^2$.
Consider the spectral sequences on the deformation complexes $\Def(\Phi)$, $\Def(\Psi)$ induced by the arity filtration.
On the $E^1$ and $E^2$ pages they agree, with the $E^2$ page being the hairy graph homology.
The spectral sequence for $\Def(\Psi)$ abuts there.
If $\Phi$ and $\Psi$ were quasi-isomorphic this would imply by Lemma \ref{lem:E2abutment}  and Remark~\ref{rem:any_resol} below that the spectral sequence for $\Def(\Phi)$ would also abut on the second page. However, as the proof of Theorem \ref{thm:PhiDef} shows that it does not, the differential is the bracket with the tripod class.
Hence we conclude that $\Phi$ and $\Psi$ cannot be quasi-isomorphic operad maps.
\hfill\qed

\subsection{Some homotopy theoretic lemmas}

\begin{lemma}\label{lem:hoenmap}
Let $\op C$, $\op C'$ be coaugmented cooperads quasi-isomorphic to $\e_n^\vee$. 
Suppose $\phi:\Omega(\op C)\to \Omega(\op C')$ is a quasi-isomorphism. 
Consider the maps of complexes
\[
 F_r \colon (\op C(r)[1], d_{\op C}) \to (\Omega(\op C)(r), d_{\op C}) \stackrel{\phi}{\lo} (\Omega(\op C')(r),d_{\op C'}) \to (\op C'(r)[1],d_{\op C'}),
\]
where the right hand map projects onto trees with only one node.
Then $F_r$ is a quasi-isomorphism for each $r$.
\end{lemma}
\begin{proof}
We may consider filtrations on $\Omega(\op C),\Omega(\op C')$ by the number of nodes in trees occurring in the cobar construction.
Consider the associated spectral sequences. On the zero-th page $\phi$ induces a map
\[
(\Omega(\op C), d_{\op C}) \to (\Omega(\op C'),d_{\op C'})
\]
so that on the $E^1$ page we have a map 
\[
\phi_1: (\Omega(\e_n^\vee), d_{1}) \to (\Omega(\e_n^\vee),d_{1})
\]
where $d_1$ is the differential on $\hoe_n=\Omega(e_n^\vee)$.

We want to show the the induced maps on homology 
\[
 f_r : \e_n^\vee(r)\cong  H(\op C(r))\to H(\op C'(r))\cong \e_n^\vee(r)
\]
are isomorphisms.
For $r=2$ the statement is clear since $\phi$ is a quasi-isomorphism by assumption.
In fact, by composing $\phi_1$ with an automorphism  we may assume without loss of generality that $f_2=\mathit{id}$. (In case $n\geq 2$ it\rq{}s enough to take an automorphism rescaling product and bracket. In case $n=1$ one applies an automorphism rescaling the product and if necessary the one reversing the product to the opposite one.)
For higher $r$, we will assume inductively that we have shown that $f_j=\mathit{id}$ for $j=2,\dots,r-1$.
Let $x\in \e_n^\vee(r)[1] \subset \hoe_n(r)$ be given.
Then $f_r(x)\in \e_n^\vee(r)[1]\subset \hoe_n(r)$ has to satisfy
\[
 d f_r(x) = dx
\]
where $d$ is the differential in $\hoe_n$.
In other words the element $f_r(x)-x$ is $d$-closed, and hence is sent to zero by all cooperadic cocompositions.
However, since all cogenerators of $e_n^\vee$ are located in arity $2$, this implies that $f_r(x)-x=0$ for $r\geq 3$.
Since $x$ was arbitrary, the statement follows.
\end{proof}


\begin{lemma}\label{lem:E2abutment}
 Consider a commutative diagram of the form 
 \[
  \begin{tikzpicture}
   \matrix[diagram](m) { \op P & \op P' \\ \op Q & \op Q' \\ };
   \draw[-latex] (m-1-1) edge node[auto] {$\phi$} (m-1-2) edge node[left] {$f$} (m-2-1)
        (m-2-1) edge node[auto] {$\psi$}  (m-2-2)
        (m-1-2) edge node[auto] {$g$} (m-2-2);
  \end{tikzpicture}
 \]
where the horizontal arrows are quasi-isomorphisms, and where $\op P$ is quasi-isomorphic to $\e_n$.
Consider the two spectral sequences associated to the arity filtrations on $\Def(\Omega(B(\op P))\stackrel{f}{\lo}  \op Q)$ and $\Def(\Omega(B(\op P'))\stackrel{g}{\lo} \op Q')$.
Then one of the two spectral sequences abuts at the $E^2$ page if and only if so does the other.
\end{lemma}
\begin{proof}
We consider the zigzag of quasi-isomorphisms
\[
 \Def(\Omega(B(\op P))\stackrel{f}{\lo} \op Q) \to \Def(\Omega(B(\op P))  \stackrel{\phi\circ g}{\lo} \op Q')  \stackrel{p}{\longleftarrow}\Def(\Omega(B(\op P'))  \stackrel{g}{\lo} \op Q').
\]
On all three deformation complexes there are filtrations by arity and the maps are compatible with these filtrations.
On the $E^0$ page we consider complexes
\[
 \prod_r \Hom_{S_r}(B(\op P)(r),\op Q(r))   \to \prod_r \Hom_{S_r}(B(\op P)(r),\op Q'(r)) \leftarrow \prod_r \Hom_{S_r}(B(\op P')(r),\op Q'(r))
\]
with the differential being induced by that of $B(\op P)$, $B(\op P')$ and $\op Q$, $\op Q'$.
On the $E^1$ page we hence consider a zig-zag
\[
 \prod_r \Hom_{S_r}(e_n^\vee(r), H\op Q(r))   \to \prod_r \Hom_{S_r}(e_n^\vee(r),H\op Q'(r)) \leftarrow \prod_r \Hom_{S_r}(e_n^\vee(r),H\op Q'(r)).
\]
The left hand arrow is an isomorphism since so is the map $H(\psi)$ by assumption.
 The right hand arrow is an isomorphism by Lemma~\ref{lem:hoenmap}. It follows that all complexes are quasi-isomorphic.

It also follows that the induced morphisms on the $E^2$ page of the spectral sequences are isomorphisms as well, and hence if one of the spectral sequences abuts at this page, so have to do the others.
\end{proof}

\begin{rem}\label{rem:any_resol}
In the above proof it is inessential that we took the bar-cobar resolutions of $\op P$, $\op P'$. In fact, one can take any other cofibrant resolution of the form 
 $\Omega(C)\twoheadrightarrow \op P$, $\Omega(\op C')\twoheadrightarrow P'$, or no resolution if $\op P$, $\op P'$ are already of that form.
\end{rem}
%
%
%

\subsection{Remarks and  alternative arguments}\label{ss:alternative}
Another way to show that the natural inclusion $E_n\to E_{n+1}$ is not formal would be to compare the deformation homology $H(\Def(E_n \to E_{n+1}))$ with $H(\Def(\e_n\stackrel{*}{\to}\e_{n+1}))$. However it turns out that both are completed symmetric algebras generated by countably many classes in both positive and negative degrees. In the first case the negative generators are tripod and loop classes --- all other generators are positive. In the second case there are much more both negative and positive generators~\cite{TW}, but as a result the total homology in both cases are infinite dimensional vector spaces in every degree, and thus as graded vector spaces are indistinguishable.  Thus to show that the deformation homology is different, one has to use the algebraic structures on this homology, but also natural topology that is necessary to define the completed tensor product, completed cup product, and the space of primitives with respect to the completed cup-product.

  Another approach would be to look at  the maps 
\begin{equation}
\label{equ:HDefmap}
H(\hDer_*(E_{n+1})) \to H(\Def(E_n \to E_{n+1}))[1],
\end{equation}
which are very different depending on whether on the right hand side one deforms the natural map or the "trivial" map $E_n \stackrel{*}{\lo} E_{n+1}$ factoring through $\Com$. 
Concretely, the scaling and hedgehog classes in $H(\hDer_*(E_{n+1}))$
are in the kernel of the map to  $H(\Def(E_n \stackrel{*}{\lo} E_{n+1}))$~\cite{TW}, while
they are sent to non-zero classes in $H(\Def(E_n \to E_{n+1}))$.
Namely, the scaling class is sent to the tripod class, and the loop classes are sent to the hedgehog classes.
A subtlety in this approach is that the action of the homotopy automorphism of $E_{n+1}$ on the scaling and loop classes can move those to some elements which would not be any more in the kernel of~\eqref{equ:HDefmap}. So, one would need to show that there will still be classes in the kernel with similar leading terms.


\section{Unitary version of the results}
Above we considered $E_n$ operads as having no operations of arity zero.
However, the topological little cubes operads are naturally endowed with a zero-ary operation, insertion of which amounts to forgetting the cube inserted into.
Similarly, the operads $\e_n$, $\Graphs_n$ etc. have natural extensions having an operation in arity zero which we denote by $\e_n^\bbo$, $\Graphs_n^\bbo$ etc.
Concretely,
\begin{align*}
\e_n^\bbo(N) &= 
\begin{cases}
\e_n(N) &\text{for $N\geq 1$} \\
\K\, \bbo &\text{for $N=0$}
\end{cases}
&
\Graphs_n^\bbo(N) &= 
\begin{cases}
\Graphs_n(N) &\text{for $N\geq 1$} \\
\K\, \bbo &\text{for $N=0$}
\end{cases}.
\end{align*}
The operadic compositions are extended as follows. The zero-ary operation $\bbo$ is killed by the bracket in $\e_n$ and is a unit with respect to the product, i.e.,
\begin{align}\label{equ:uen relations}
[\cdot, \bbo] &= 0 & \cdot\wedge \bbo &= \mathit{Id}_{op} 
\end{align}
where $\mathit{Id}_{op}\in \e_n(1)$ is the operadic unit.

Similarly, inserting the zero-ary operation $\bbo\in \Graphs^\bbo(0)$ into a vertex $j$ of a graph $\Gamma\in \Graphs(N)$ forgets that vertex if it has valence zero, and maps the graph to zero otherwise.

Note that the action of $\GC_n$ on $\Graphs_n$ by operadic derivations naturally extends to $\Graphs_n^\bbo$.

Note also that M. Kontsevich's proof of the formality of the little $n$-cubes operad outlined in section \ref{sec:formality review} naturally extends to the unital case to produce a zig-zag of quasi-isomorphisms
\[
 C(\FM_n^\bbo) \to \Graphs_n^\bbo \leftarrow \e_n^\bbo.
\]

Note also that our proof of Theorem \ref{thm:main} for $k\geq 2$ given in section \ref{sec:proof k geq 2} is independent of the presence or non-presence of operations in arity zero.
Furthermore, the statement that the map $E_n\to E_{n+1}$ is not formal is stronger then that of $E_n^\bbo\to E_{n+1}^\bbo$ being non-formal.
Hence we arrive at the following unital version of Theorem \ref{thm:main}.

\begin{cor}[Unital version of Theorem \ref{thm:main}]
Over $\R$, the operad maps $E_n^\bbo\to E_{n+k}^\bbo$ are formal for $k\geq 2$ and non-formal for $k=1$.
\end{cor}

One may also consider the deformation complexes of the operad maps $E_n^\bbo\to E_{n+k}^\bbo$.
To this end, one in particular needs a tractable quasi-free algebraic model for $E_n^\bbo$.
This is provided by the Koszul duality theory of operads with zero-ary operations developed by Hirsh and Mill\`es \cite{HirshMilles}.
Let us recall the relevant statements from loc. cit., in the special case we are interested in. One defines the cooperad $(\e_n^\bbo)^\vee$ as the cooperad cogenerated by $\e_n^\vee$ and an additional operation of arity zero and degree $1$, and no additional corelations.
Concretely, we may identify 
\[
(\e_n^\bbo)^\vee(N) \cong \bigoplus_{k\geq 0} \left( \e_n^\vee(N+k)\otimes \K[-1]^{\otimes k}\right)_{\bbS_k}.
\]
Following the theory of \cite{HirshMilles} (applied to the case $\e_n^\bbo$) we may construct a quasi-free resolution of $\e_n^\bbo$ as 
\[
\hoe_n^\bbo := (\Omega((\e_n^\bbo)^\vee), d_\Omega + d_{\bbo})
\]
where $d_\Omega$ is the usual differential on the bar construction and $d_{\bbo}$ is an additional piece ensuring that the relations \eqref{equ:uen relations} hold in homology.

Suppose we are given an operad map $\hoe_n^\bbo\to \Graphs_{n+k}$. 
We may form the complex of derivations of that map
\begin{align*}
\Der(\hoe_n^\bbo\to \Graphs_{n+k}^\bbo)
&\cong 
\prod_{N\geq 0} \Hom_{\bbS_N}( \overline{ (\e_n^\bbo)^\vee}(N)[1], \Graphs_{n+k}^\bbo(N))
\\ &\cong \left(\prod_{\substack{ N, r \geq 0 \\ (N,r)\neq (1,0) }} \left( \e_n\{-n\}(N+r)\otimes \K[1]^{\otimes r}  \otimes \Graphs_{n+k}^\bbo(N) \right)^{\bbS_N\times \bbS_r}\right)[-1]
\end{align*}

The homology of this complex will contain one class corresponding to the "trivial" derivation given by rescaling by $(\text{arity}-1)$. We kill this one class by passing to the reduced deformation complex
\[
\Der_*(\hoe_n^\bbo\to \Graphs_{n+k}^\bbo)
:= 
\left(\prod_{\substack{ N, r \geq 0 \\ N+r\geq 2 }} \left( \e_n\{-n\}(N+r)\otimes \K[1]^{\otimes r}  \otimes \Graphs_{n+k}^\bbo(N) \right)^{\bbS_N\times \bbS_r}\right)[-1],
\]
whose suspension we denote by $\Def(\hoe_n^\bbo\to \Graphs_{n+k}^\bbo)$.\footnote{Notice that passing to the reduced version of the complex of derivations  in the unital setting is equivalent to saying   that we consider deformations that keep fixed  the arity zero operation corresponding to forgetting a cube in a configuration.}
Note that there is a map of operads
\[
\hoe_n\to \hoe_n^\bbo
\]
and hence a map of deformation complexes 
\[
\Der(\hoe_n^\bbo\to \Graphs_{n+k}^\bbo) \to \Der(\hoe_n\to \Graphs_{n+k}^\bbo) \cong \Der(\hoe_n\to \Graphs_{n+k}).
\]
These maps in turn yield a map
\[
\Def(\hoe_n^\bbo\to \Graphs_{n+k}^\bbo) \to  \Def(\hoe_n\to \Graphs_{n+k}).
\]

The following proposition shows that (for our purposes) the deformation theory in the unital case is identical to that in the non-unital case considered above.

\begin{prop}\label{pr:unital}
Suppose the operad map $\hoe_n^\bbo\to \Graphs_{n+k}^\bbo$ is a  quasi-isomorphism in arity zero and that the product generator is sent to a non-zero multiple of the graph with two vertices and no edge.
Then the map 
\[
\Def(\hoe_n^\bbo\to \Graphs_{n+k}^\bbo) \to  \Def(\hoe_n\to \Graphs_{n+k})
\]
is a quasi-isomorphism.
\end{prop}
\begin{proof}
We will consider a descending complete filtration on 
\[
\Def(\hoe_n^\bbo\to \Graphs_{n+k}^\bbo)
\cong
\prod_{\substack{ N, r \geq 0 \\ N+r\geq 2 }} 
\left( \e_n\{-n\}(N+r)\otimes \K[1]^{\otimes r}  \otimes \Graphs_{n+k}^\bbo(N) \right)^{\bbS_N\times \bbS_r}
\]
by the quantity $N+r$. The differential on the associated graded consists of two pieces
\[
d_0 = \delta + d'
\]
where $\delta$ is the differential on $\Graphs_{n+k}^\bbo$ while 
\[
d' : \left( \e_n\{-n\}(N+r)\otimes \K[1]^{\otimes r}  \otimes \Graphs_{n+k}^\bbo(N) \right)^{\bbS_N\times \bbS_r}
\to 
\left( \e_n\{-n\}(N+r)\otimes \K[1]^{\otimes (r+1)}  \otimes \Graphs_{n+k}^\bbo(N-1) \right)^{\bbS_{N-1}\times \bbS_{r+1}}
\]
acts by inserting the zero-ary operation in one vertex of the second factor.
Note that this insertion is zero unless the vertex has valence 0.
We may split any graph in $\Graphs_{n+k}^\bbo(N)$ into a piece all of whose external vertices have positive valence, and possibly several vertices with valence zero.
Let us call the subspace of graphs all of whose external vertices have positive valence $\Graphs_{n+k, norm}^\bbo(N)$. Then
\[
\left( \e_n\{-n\}(N+r)\otimes \K[1]^{\otimes r}  \otimes \Graphs_{n+k}^\bbo(N) \right)^{\bbS_N\times \bbS_r}
= \bigoplus_{s=0}^N
\left( \e_n\{-n\}(N+r)\otimes \K[1]^{\otimes r}  \otimes \Graphs_{n+k, norm}^\bbo(N-s) \right)^{ \bbS_{N-s}\times \bbS_r\times \bbS_s}
\]
Note that the differential $d'$ preserves the quantity $r+s$. Furthermore, it is easy to see that the subcomplex of fixed $r+s$ is acyclic if $r+s>0$.
It is easy to see from these considerations that the homology of the associated graded in our spectral sequence is 
\[
E^1 = \prod_{N\geq 2} \left( \e_n\{-n\}(N)\otimes \K[1]^{\otimes r}  \otimes \e_{n+k, norm}^\bbo(N) \right)^{\bbS_r\times \bbS_N}.
\]
The next differential increases the quantity $N$ by 1. 

Note that the resulting complex is quasi-isomorphic to the normalized subcomplex of the $E^1$ page of the spectral sequence associated to the filtration by arity ($N$) of the non-unital deformation complex $\Def(\hoe_n\to \Graphs_{n+k})$.
But it is well-known that the normalized complex is quasi-isomorphic to the full complex, and hence the Proposition follows by standard spectral sequence arguments.
\end{proof}


As a consequence of Lemma~\ref{lem:Phi_m2} and Proposition~\ref{pr:unital} we get:

\begin{cor}[Unital version of Theorems \ref{thm:schoenflies} and \ref{thm:PhiDef}]
\begin{multline*}
 H(\Def(\hoe_{n-1}^\bbo\to \Graphs_{n}^\bbo)) 
 \cong S^+\left(\Bigl(H(\GC_n^2) \oplus \K T\Bigr)[n]\right)[1-n]
 \cong \\
 S^+\left(\Bigl(H(\GC_n) \oplus \prod_{1\leq r\equiv 2n+1 \text{ mod $4$}}\K H_r \oplus \K T\Bigr)[n]\right)[1-n]
\end{multline*}
where the $H_r$ denote the hedgehog classes and $T$ the tripod class.
\end{cor}

\section{Connection to embedding calculus}\label{s:emb_calcl}
The manifold calculus developed by Goodwillie and Weiss~\cite{GoodWeiss,WeissEmb} has been shown to be deeply connected to the theory of operads~\cite{ALV,Turchin2,DeBrito-Weiss,Sinha,Tur_FM}. One of the main applications of this calculus is the study of embedding spaces. Let $Emb_\partial(D^m,D^n)$ denote the space of smooth embeddings $D^m\hookrightarrow D^n$ pointwise fixed at the boundary as some equatorial inclusion. In particular for $m=n$ we get the space $\Diff_\partial(D^n)$ of diffeomorphisms of a disc preserving the boundary pointwise. By taking derivative at every point we get an obvious map
$$
Emb_\partial(D^m,D^n)\to \Omega^mV_{m,n},
$$
where $V_{m,n}$  is the space of linear injective maps $\R^m\hookrightarrow \R^n$. The homotopy fiber $\overline{Emb}_\partial(D^m,D^n)$ of this map (over the base point) is usually called {\it space of embeddings modulo immersions}. We denote this space  by $\overline{\Diff}_\partial(D^n)$ in case $m=n$. In order to study the homology of $\overline{Emb}_\partial(D^m,D^n)$ from the point of view of the calculus, one defines a cofunctor 
$$
C_*\overline{Emb}_\partial(-,D^n)\colon \widetilde{\op O}(D^m)\to Ch
$$
from the category of open subsets of $D^m$ containing $\partial D^m$ to the category of chain complexes. This cofunctor assigns to any open set $U\subset D^m$ the chain complex $C_*\overline{Emb}_\partial(U,D^n)$, where $\overline{Emb}_\partial(U,D^n)$ is a similar space of embeddings modulo immersions. The general machinery of the calculus provides us with a map
$$
C_*\overline{Emb}_\partial(D^m,D^n)\to T_\infty C_*\overline{Emb}_\partial(D^m,D^n),
$$
where the right-hand side is the limit of the Goodwillie-Weiss tower. This map is a quasi-isomorphism only for $n\geq 2m+2$, i.~e. when the space $\overline{Emb}_\partial(D^m,D^n)$ is connected. Even when the convergence does not hold, one can still study the limit, that might produce interesting cohomology classes and invariants of spaces of knots. In particular, one can also apply this construction to the case $m=n$, or in other words to the study of the space of diffeomorphisms of $D^n$ fixing boundary pointwise. The following result describes this limit from the point of view of the theory of operads.

\begin{thm*}[\cite{Turchin2}]
One has a natural equivalence of complexes
$$
T_\infty C_*\overline{Emb}_\partial(D^m,D^n)\simeq 
\hIBim_{E_m^\bbo}(E_m^\bbo,E_n^\bbo).
$$
\end{thm*}

This result holds for any ring of coeffcients. The right-hand side is the space of derived morphisms $E_m^\bbo\to E_n^\bbo$ in the category of infinitesimal bimodules over $E_m^\bbo$.  To make things precise  we will be interested in the limit $T_\infty C_*\overline{Emb}_\partial(D^m,D^n)$ in the model category of unbounded chain complexes.\footnote{One can also consider such limit in the category of   non-negatively graded  complexes. The latter one is obtained from the former one by a simple truncation preserving non-negatively graded homology.}  We conjecture that for any coefficients,
\begin{equation}
\label{equ:conject_desusp}
\hIBim_{E_m^\bbo}(E_m^\bbo,E_n^\bbo)\simeq \hDer_*(E_m^\bbo\to E_n^\bbo)[m+1]
\simeq \hDer_*(E_m\to E_n)[m+1].
\end{equation}
(To be precise the left-hand side must be quotiented by a one-dimensional vector space in degree zero, or in other words one needs to take the limit of the {\it reduced} singular chains $T_\infty \tilde C_*\overline{Emb}_\partial(D^m,D^n)$.) Notice also that in case of characteristic zero the second equivalence is proved by Proposition~\ref{pr:unital}.) More generally we have the following 

\begin{con*}
For any commutative ring of coefficients $\mathbb K$, given a morphism of differential graded  operads  $E_m^\bbo\to {\op O}$ with $\op O$ augmented over $\Com^\bbo$ and doubly reduced: ${\op O}(0)={\op O}(1)={\mathbb K}$,  one has
\begin{equation}
\label{equ:conject_desusp_general}
\hIBim_{E_m^\bbo}(E_m^\bbo,{\op O})\simeq \hDer_*(E_m^\bbo\to {\op O})[m+1]\simeq \hDer_*(E_m\to {\op O})[m+1].
\end{equation}
\end{con*}
%
The equivalence~\eqref{equ:conject_desusp_general} has been shown for $\mathbb K=\R$ in the case when the map $E_m^\bbo\to {\op O}$ factors through $\Com^\bbo$, see~\cite{Turchin3}, and also in the case $m=1$ (see the comment right after the theorem below). The topological version of this conjecture  has been proved by Dwyer and Hess:

\begin{thm*}[\cite{DwHe2}]
Given a morphism of topological operads $\calC_m^\bbo\stackrel{f}{\to}{\op O}$ with $\op O$ being doubly reduced: ${\op O}(0)\simeq {\op O}(1)\simeq *$, one has a weak equivalence of spaces
$$
\hIBim_{\calC_m^\bbo}(\calC_m^\bbo,{\op O})\simeq \Omega^{m+1}\hOperad(\calC_m^\bbo, {\op O}),
$$
where the space $\hOperad(\calC_m^\bbo, {\op O})$ is the space of derived maps of operads $\calC_m^\bbo\to {\op O}$ based at~$f$.
\end{thm*}

In case $m=1$ this theorem was proved in~\cite{DwHe1,Tur_deloop}. The proof in~\cite{Tur_deloop} uses an explicit cellular cofibrant replacement of $\calC_1^\bbo$ and thus can be adjusted to the differential graded context.

Given $m_1<m_2\leq n$, one gets natural maps
\begin{equation}
\label{equ:hDer_map}
\hDer_*(E_{m_2}\to E_n)\to \hDer_*(E_{m_1}\to E_n).
\end{equation}
For example~\eqref{equ:HDefmap} is a particular case of such map. We believe that~\eqref{equ:hDer_map} models the map
$$
T_\infty \tilde C_*\Sigma^{m_2-m_1}\overline{Emb}_\partial(D^{m_2},D^n)\to
T_\infty \tilde C_*\overline{Emb}_\partial(D^{m_1},D^n)
$$
obtained from the natural scanning map $\overline{Emb}_\partial(D^{m_2},D^n)\to\Omega^{m_2-m_1}
\overline{Emb}_\partial(D^{m_1},D^n)$ applying adjunction between loops and suspensions and then $T_\infty \tilde C_*$.

Recall that the Cerf Lemma \cite[Appendix, Section~5, Proposition~5]{Cerf2}, \cite[Proposition~5.3]{Budney} says that the natural scanning map~\eqref{eq:cerf} is a weak equivalence. Since the Stiefel manifold  $V_{n-1,n}$ is just the group $\mathit{SO}(n)$, one gets that the modulo immersions scanning map 
\begin{equation}
\overline{\Diff}_\partial(D^{n+1})\to \Omega \overline{Emb}_\partial(D^n,D^{n+1})
\label{eq:cerf_bar}
\end{equation}
is also an equivalence.
 Modulo the Conjecture above, our Algebraic Cerf Lemma (Theorem~\ref{thm:PhiDef}) implies  that the real homology of $T_\infty C_* \overline{Emb}_\partial(D^n,D^{n+1})$ and of $T_\infty C_* \overline{\Diff}_\partial(D^{n+1})$ are  completed symmetric algebras whose spaces of generators  up to a shift in degree by one are the same. (Note however that the generators could be of both positive and negative degrees.) Recall also that $\pi_0 Emb_\partial(D^n,D^{n+1})$ is known to be a torsion group for $n\neq 3,\, 4$~\cite[Section~5]{Budney}, and therefore is rationally trivial.  This is actually a consequence of the (generalized) Schoenflies theorem which has been proved for all dimensions except $n=3$ and which states that any smoothly embedded sphere $S^n$ in $S^{n+1}$ bounds a smooth disc on each side. As a consequence the natural map $\pi_0 \Diff_\partial (D^{n})\to\pi_0 Emb_\partial(D^n,D^{n+1})$ is surjective. In fact it follows from the Cerf pseudoisotopy theorem that this map is an isomorphism for $n\neq 3,\,4$ and the corresponding group is the group $\theta^{n+1}$ of exotic smooth structures on $S^{n+1}$~\cite[Section~5]{Budney}.

\bibliographystyle{plain}

\begin{thebibliography}{10}


\bibitem{ALV} Gregory Arone, Pascal Lambrechts, and I. Voli\`c. 
\newblock{Calculus of functors, operad formality, and rational homology of embedding spaces.}
\newblock {\em Acta Math.} 199 (2007), no. 2, 153--198.


\bibitem{Turchin3}
Gregory Arone and Victor Turchin.
\newblock {Graph-complexes computing the rational homotopy of high dimensional
  analogues of spaces of long knots}.
\newblock To appear in {\em Annales de l'Institut Fourier} 64 (2014),
\newblock arXiv:1108.1001.



\bibitem{Turchin2}
Gregory Arone and Victor Turchin.
\newblock {On the rational homology of high dimensional analogues of spaces of
  long knots}.
\newblock {\em Geom. Topol.} 18 (2014) 1261--1322.


\bibitem{DeBrito-Weiss} 
Pedro  Boavida de Brito and Michael Weiss.
\newblock{Manifold calculus and homotopy sheaves.} 
\newblock  {\em Homol. Homot. Appl.} 
\newblock  15 (2013), no. 2, 361–-383. 

\bibitem{Budney} Ryan Budney. 
\newblock{A family of embedding spaces.}
\newblock{Groups, homotopy and configuration spaces},  41--83,
\newblock \emph{ Geom. \&  Topol. Monogr.}
\newblock 13, Geom. Topol. Publ., Coventry, 2008.



\bibitem{Cerf1}  Jean Cerf.
\newblock{Th\'eor\`emes de fibration des espaces de plongements. Applications.} 
\newblock \emph{S\'em. H. Cartan}, 
\newblock v. 15 no. 8 (1962--63), 1--13. 

\bibitem{Cerf2} Jean Cerf. 
\newblock{Sur les diff\'eomorphismes de la sph\`ere de dimension trois ($\Gamma_4=0$).}
\newblock  (French) \emph{Lecture Notes in Mathematics}, No. 53 
\newblock Springer-Verlag, Berlin-New York 1968 xii+133 pp.


\bibitem{DwHe1} William  Dwyer and Kathryn  Hess.
\newblock Long knots and maps between operads.
\newblock {\em Geom. Topol.} 16 (2012), no. 2, 919–-955.

\bibitem{DwHe2} W.  Dwyer, K.  Hess. Paper to appear.

\bibitem{CW}
Damien Calaque and Thomas Willwacher.
\newblock {Triviality of the higher Formality Theorem}.
\newblock 2013.
\newblock arXiv:1310.4605.

\bibitem{Costello}
Jim Conant, Jean Costello, Victor Turchin, and Patrick Weed.
\newblock {Two-loop part of the rational homotopy of spaces of long embeddings.}
\newblock {\em J. Knot Theory Ram.} 23 (4), 2014.

\bibitem{vastwisting}
Vasily Dolgushev and Thomas Willwacher.
\newblock {Operadic twisting - with an application to Deligne's conjecture},
  2012.
\newblock  To appear in J. Pure Appl. Alg., arXiv:1207.2180.


\bibitem{Fresse} Benoit Fresse.
\newblock{Homotopy of Operads and Grothendieck-Teichm\"uller Groups}.
\newblock{Book in preparation.}

\bibitem{GoodWeiss} Thomas  Goodwillie and Michael Weiss.
\newblock{Embeddings from the point of view of immersion theory.~II.}
\newblock  {\em Geom. Topol.} 3 (1999), 103-–118.


\bibitem{HLTV}
Robert Hardt, Pascal Lambrechts, Victor Turchin, and Ismar Voli{\'c}.
\newblock Real homotopy theory of semi-algebraic sets.
\newblock {\em Algebr. Geom. Topol.}, 11(5):2477--2545, 2011.

\bibitem{HirshMilles}
Joseph Hirsh and Joan Mill\`es.
\newblock Curved Koszul duality theory.
\newblock {\em Mathematische Annalen} 354(4):1465--1520, 2012.

\bibitem{Khovanskii}
Askol\rq{}d G. Khovanskii. 
\newblock  On a lemma of Kontsevich.
\newblock  (Russian) {\em Funktsional.
Anal. i Prilozhen.} 31 (1997), no. 4, 89--91; 
\newblock translation in {\em Funct. Anal.
Appl.} 31 (1997), no. 4, 296–-298 (1998).


\bibitem{K1}  Maxim Kontsevich.
\newblock Deformation quantization of Poisson manifolds.
\newblock {\em  Lett.
Math. Phys.} 66 (2003), no. 3, 157-–216.

\bibitem{K2}
Maxim Kontsevich.
\newblock Operads and {M}otives in {D}eformation {Q}uantization.
\newblock {\em Lett. Math. Phys.} 48 (1999), 35--72.

\bibitem{KS} Maxim Kontsevich and Yan Soibelman.
\newblock Deformations of algebras over operads and the Deligne conjecture. 
\newblock In {\em Conf\'erence Mosch\'e Flato 1999, Vol. I (Dijon)}, 
\newblock volume 21 of {\em Math. Phys. Stud.,} 225--307, Kluwer Acad. Publ., Dordrecht, 2000.

\bibitem{LambrechtsTurchin}
Pascal Lambrechts and Victor Turchin.
\newblock Homotopy graph-complex for configuration and knot spaces.
\newblock {\em Trans. Amer. Math. Soc.}, 361(1):207--222, 2009.

\bibitem{LV}
Pascal Lambrechts and Ismar Voli\'c.
\newblock  Formality of the little $N$-disks operad. 
\newblock {\em Mem. Amer. Math. Soc. } 
\newblock 230 (2014), no. 1079, viii+116 pp. ISBN: 978-0-8218-9212-1 .   

\bibitem{livernet}
Muriel Livernet.
\newblock Non-formality of the Swiss-Cheese operad, 2014.
\newblock arXiv:1404.2484

\bibitem{lodayval}
J.-L. Loday and B.~Vallette.
\newblock {\em {Algebraic Operads}}.
\newblock Number 346 in {Grundlehren der mathematischen Wissenschaften}.
  {Springer, Heidelberg}, {2012}.

\bibitem{Sinha} Dev Sinha. 
\newblock Operads and knot spaces.  
\newblock {\em   J. Amer. Math. Soc.} 19 (2006), no. 2, 461–-486.

\bibitem{tamenaction} Dmitry Tamarkin.
  \newblock Deformation complex of a $d$-algebra is a $(d+1)$-algebra.
  \newblock math.QA/0010072, 2000.

\bibitem{Turchin1}
Victor Turchin.
\newblock Hodge-type decomposition in the homology of long knots.
\newblock {\em J. Topol.}, 3(3):487--534, 2010.

\bibitem{Tur_FM}  Victor Turchin.
\newblock  Context-free manifold calculus and the Fulton-MacPherson operad.
\newblock {\em Algebr. Geom. Topol. } 13 (2013), no. 3, 1243–-1271.

\bibitem{Tur_deloop} Victor Turchin.
\newblock Delooping totalization of a multiplicative operad.
\newblock {\em J. Homotopy Relat. Struct.} 2013. DOI
10.1007/s40062-013-0032-9.

\bibitem{TW}
Victor Turchin and Thomas Willwacher.
\newblock{Relative deformation theory of little cubes operads},
\newblock in preparation.

\bibitem{WeissEmb} Michael Weiss.
\newblock{Embeddings from the point of view of immersion theory. I.}
\newblock  {\em Geom. Topol.} 3 (1999), 67--101.

\bibitem{grt}
Thomas Willwacher.
\newblock {M. Kontsevich's graph complex and the Grothendieck-Teichm\"uller Lie
  algebra}, 2010.
\newblock To appear in Invent. Math., arxiv:1009.1654.

\end{thebibliography}

\end{document}